\documentclass[11pt]{amsart}

\usepackage{amsmath}
\usepackage{amssymb}
\usepackage{mathrsfs}
\usepackage[all]{xy}
\usepackage{tikz-cd}

\newtheorem{mtheorem}{Main Theorem}[section]

\newtheorem{theorem}{Theorem}[section]
\newtheorem{claim}[theorem]{Claim}

\newtheorem{lemma}[theorem]{Lemma}

\newtheorem{proposition}[theorem]{Proposition}
\newtheorem{corollary}[theorem]{Corollary}

\theoremstyle{definition}
\newtheorem{definition}[theorem]{Definition}
\newtheorem{example}[theorem]{Example}

\newtheorem{question}[theorem]{Question}

\theoremstyle{remark}
\newtheorem{remark}[theorem]{Remark}

\newtheorem{fact}[theorem]{Fact}

\def\l{{\langle}}
\def\r{{\rangle}}

\newcount\skewfactor
\def\mathunderaccent#1#2 {\let\theaccent#1\skewfactor#2
\mathpalette\putaccentunder}
\def\putaccentunder#1#2{\oalign{$#1#2$\crcr\hidewidth
\vbox to.2ex{\hbox{$#1\skew\skewfactor\theaccent{}$}\vss}\hidewidth}}

\newcommand{\Gal}[1]{\mathrm{Gal}({#1})}

\def\smallbox#1{\leavevmode\thinspace\hbox{\vrule\vtop{\vbox
   {\hrule\kern1pt\hbox{\vphantom{\tt/}\thinspace{\tt#1}\thinspace}}
   \kern1pt\hrule}\vrule}\thinspace}

\DeclareMathOperator{\crit}{crit}

\newcommand{\cf}{{\rm cf}}

\title{The Galvin Property under the Ultrapower Axiom}
\author[Benhamou]{Tom Benhamou}
\address[Benhamou]{Department of Mathematics, Rutgers University, Piscataway (NJ) 08854-
8019, USA.}
\email{tom.benhamou@rutgers.edu}
\thanks{The research of the first author was supported by the National Science Foundation under Grant
No. DMS-2346680}
\author{Gabriel Goldberg}
\address[Goldberg]{Department of Mathematics, University of California, Berkeley, Berkeley, CA 94720-3840 USA }
\email{ggoldberg@berkeley.edu}

\begin{document}

\begin{abstract}
We continue the study of the Galvin property from \cite{bgs} and \cite{SatInCan}. In particular, we deepen the connection between certain diamond-like principles and non-Galvin ultrafilters. We also show that any Dodd sound non p-point ultrafilter is non-Galvin. 
We use these ideas to formulate
what appears to be the optimal large cardinal hypothesis implying the existence of a non-Galvin ultrafilter, improving on a result from \cite{TomNatasha}. Finally, we use a strengthening of the Ultrapower Axiom to prove that in all the known canonical inner models, a $\kappa$-complete ultrafilter has the Galvin property if and only if it is an iterated sum of $p$-points.  
\end{abstract}
\maketitle

\section{Introduction}
In this paper, we study certain aspects of the \textit{Galvin property} of ultrafilters:
\begin{definition}
    Let $U$ be a uniform ultrafilter over $\kappa$. We say that $U$ has the Galvin property if for any sequence $\l A_i\r_{ i<2^\kappa}$, there is $I\in[ 2^\kappa]^{\kappa}$ such that $\bigcap_{i\in I}A_i\in U$. 
\end{definition}
More generally, if $\lambda\leq \kappa$ and $U$ is a uniform ultrafilter over $\kappa$, we denote by $\text{Gal}(U,\lambda,2^\kappa)$ the statement that for any $\l A_i\r_{i<2^\kappa}$ there is $I\in[2^\kappa]^\lambda$ such that $\bigcap_{i\in I}A_i\in U$. 
Galvin proved in 1973 every normal ultrafilter has the Galvin property. Gitik and Benhamou \cite{TomMoti} recently improved this result to show that any product of $\kappa$-complete $p$-points over $\kappa$ has the Galvin property.\footnote{A $\kappa$-complete ultrafilter $U$ over $\kappa$ is called $p$-point if every sequence $\l A_i\mid i<\kappa\r\subseteq U$ has a measure-one pseudo-intersection; that is, there is $A\in U$ such that for every $i<\kappa$, $|A\setminus A_i| < \kappa$.}
Benhamou \cite{SatInCan} then proved what appears to be a slight improvement of this result:
\begin{theorem}\label{Theorem: Galvin Improvment}
    Suppose that $U$ is Rudin-Keisler equivalent to an $n$-fold sum of $\kappa$-complete $p$-points (See Definition \ref{Definition: sums and limits}). Then $U$ has the Galvin property. 
\end{theorem}
The main theorem of this paper shows
that under natural combinatorial hypotheses which hold in all known canonical inner models, the converse of the above theorem is true.
\begin{mtheorem}\label{mthorem: kappa complete galvin characterization}
    Assume the Ultrapower Axiom
    and that every irreducible ultrafilter is Dodd sound. If
    \(U\) is a \(\kappa\)-complete ultrafilter on \(\kappa\)
    with the Galvin property, then
    \(U\) is Rudin-Keisler equivalent to an iterated sum of 
    \(\kappa\)-complete p-points on \(\kappa\).
\end{mtheorem}
The hypotheses of this theorem will be discussed and explained further later in the introduction.

The study of the Galvin property is motivated by its presence in
various areas of set theory and infinite combinatorics \cite{TomMoti,partOne,bgs,Non-GalvinFil,GalDet,TomNatasha,ghhm,bgp}. One particularly noteworthy incarnation of the Galvin property is the maximal class in the Tukey order, which we shall now explain in more detail.
\begin{definition}
    For two posets $(P,\leq_P),(Q,\leq_Q)$\footnote{We shall abuse notation by suppressing the order in a poset.}, we say that $P\leq_T Q$ if there is a cofinal map $f:Q\rightarrow P$.\footnote{A map $f:Q\rightarrow P$ is called cofinal if for every cofinal set $B\subseteq Q$, $f''B$ is cofinal in $P$.} We say that $P,Q$ are Tukey equivalent and denote $P\equiv_TQ$, if $P\leq_T Q$ and $Q\leq_TP$.
\end{definition}
The Tukey order finds its origins in the Moore-Smith convergence notions of nets and is of particular interest when considering the poset $(U,\supseteq)$ where $U$ is an ultrafilter. The Tukey order restricted to ultrafilters over $\omega$ has been extensively studied by Isbell \cite{Isbell65}, Milovich \cite{Milovich08,Milovich12},
Dobrinen and Todorcevic \cite{Dobrinen/Todorcevic11,Dobrinen/Todorcevic14,DobrinenJSL15}, Raghavan, Dobrinen, and Blass \cite{Raghavan/Todorcevic12,Blass/Dobrinen/Raghavan15}, and many others. 
Lately, this investigation has 
been stretched to ultrafilters over uncountable cardinals and in particular to measurable cardinals by Benhamou and Dobrinen \cite{TomNatasha}. It turns out that the Tukey order on $\sigma$-complete ultrafilters over measurable cardinal behaves differently from the one on $\omega$ and requires a new theory to be developed. One of these differences revolves around the maximal class. For a given $\lambda$, a uniform ultrafilter $U$ on $\kappa$ is called \textit{Tukey-top with respect to $\lambda$} if its Tukey class is above every $\lambda$-directed poset of size $2^\kappa$. It turns out that an ultrafilter $U$ is Tukey-top with respect to  $\lambda$ if and only if $\neg \text{Gal}(U,\lambda,2^\kappa)$. In particular, a uniform ultrafilter over \(\kappa\) 
is Tukey-top with respect to \(\kappa\)
if and only if it is non-Galvin.

Working in ZFC (with no additional set theoretic hypotheses), Isbell \cite{Isbell65} constructed  ultrafilters on $\omega$ which are non-Galvin, this construction was accomplished independently by Juh\'{a}sz \cite{Juhasz67}.  The first construction of non-Galvin ultrafilters over measurable cardinals is due to Garti, Shelah, and Benhamou \cite{bgs}, using the existence of Kurepa trees to prevent a certain ultrafilter from having the Galvin property. This connection between Kurepa trees and the Galvin property is further explored in this paper, where we define (Definition \ref{Definition: Diamond-Thin}) a diamond-like principle $\Diamond^*_{\text{thin}}(W)$, and a slight weakening (Definition \ref{Definition: W-Kurepa}) of it that ensures that an ultrafilter is non-Galvin (Lemma \ref{Lemma: Thin implies non-Galvin}). 

In \cite{TomNatasha}, Isbell's construction together with other features from \cite{SatInCan} enabled the construction of a non-Galvin ultrafilter over a $\kappa$-compact cardinal. Here we improve this initial large cardinal, isolate the notion of a \textit{non-Galvin cardinal} (Definition \ref{Definition: non-Galvin cardinal}), and prove the following:
\begin{mtheorem}\label{mthorem: non-Galvin implies existence of ultrafilter}
    Suppose that $\kappa$ is a non-Galvin cardinal then $\kappa$ carries a $\kappa$-complete ultrafilter $U$ such that $\neg\text{Gal}(U,\kappa,\kappa^+)$. In particular if in addition $2^\kappa=\kappa^+$ then $U$ is a non-Galvin ultrafilter.
\end{mtheorem}
We also prove that $\kappa$-compactness implies non-Galvinness (Theorem \ref{Theorem: Compact implies non-Galvin}), that some degree of Dodd soundness implies it (Corollary \ref{Corollary: nonppoint+DoddSound implies non-Galvin}), and that in the known canonical inner models, a $\kappa$-compact cardinal is a limit of non-Galvin cardinals (Proposition \ref{Proposition: UA implies compact above non-Galvin}).

In \cite{Parttwo}, Gitik and Benhamou noted that although the existence of a non-Galvin ultrafilter is equiconsistent with a measurable cardinal, the latter assumption (measurability) does not outright imply that there is a non-Galvin ultrafilter. More precisely, in Kunen's model $L[U]$, since every $\sigma$-complete ultrafilter is Rudin-Keisler equivalent to a power of the normal ultrafilter $U$, Theorem \ref{Theorem: Galvin Improvment} can be invoked to deduce the Galvin property for every  $\sigma$-complete ultrafilter in $L[U]$. Being the simplest example of a canonical inner model which can accommodate a measurable cardinal, the result in $L[U]$ suggests that the Galvin property, like many other combinatorial properties of ultrafilters, has a rigid form in the canonical inner models. Indeed, the result from $L[U]$ was later generalized \cite{SatInCan} to the Mitchell-Steel models $L[E]$ up to a measurable limit of superstrong cardinal\footnote{A cardinal $\kappa$ is \textit{superstrong} if there is an elementary embedding $j\colon V\rightarrow M$ with $\crit(j)=\kappa$ and $V_{j(\kappa)}\subseteq M$.} (See Theorem \ref{Theorem: Galvin Improvment}). These results in the inner models suggest the following question \cite[Question 5.1]{SatInCan}:
\begin{question}
    Is there an inner model with a non-Galvin ultrafilter? 
\end{question}

In this paper we take a more ambitious approach and work under the \textit{Ultrapower Axiom} (UA)\footnote{In this paper, we will  use the structural consequences of UA rather than UA itself, so we choose not to provide the precise statement of the axiom, which can be found in \cite{GoldbergUA}.} which is a combinatorial principle discovered by Goldberg \cite{GoldbergUA}. The advantage of UA is that with one simple axiom, which holds in all known canonical inner models, many of the usual principles are captured;
for example, the linearity of the Mitchell order and instances of GCH. More relevant for our purposes, the presence of UA imposes rigidity on the structure of ultrafilters:
\begin{theorem}[UA]\label{Theorem: Factorization into irreducibles}
    Let $W$ be a $\sigma$-complete ultrafilter. Then $W$ can be written as the $n$-fold sum of irreducible ultrafilters.\footnote{Recall the irreducible ultrafilters are those ultrafilters which are minimal in the Rudin-Frol\'ik order. Equivalently, $W$ is irreducible  if there is no ultrapower embedding $j\colon V\rightarrow M$ and an ultrafilter $U\in M$ such that $j_W=(j_U)^M\circ j$.}
\end{theorem}
In \cite{SatInCan}, this kind of characterization, together with further fine structural properties of the Mitchell-Steel extender models $L[E]$ was already used to prove the following:
\begin{theorem}
    If $L[E]$ is an iterable Mitchell-Steel model containing no superstrong cardinals, then every $\kappa$-complete ultrafilter in
    $L[E]$ has the Galvin property.
\end{theorem}
The point here is that in $L[E]$ every $\kappa$-complete ultrafilter takes the form of Theorem \ref{Theorem: Galvin Improvment} and therefore satisfies the Galvin property. 

The existence of canonical inner models with superstrong cardinals is open, though provable from widely believed conjectures:
the fine structure for inner models with
superstrong cardinals has been developed
assuming iterability hypotheses \cite{SteelBook}.
Therefore the current knowledge about canonical inner models does not quite reach the level where a $\kappa$-complete non-Galvin ultrafilter exists, although our results below show that the
conditional canonical inner models built based on iterability hypotheses can contain non-Galvin ultrafilters.

Here we shall prove the following stronger (in several senses) result:
\begin{mtheorem}[UA]\label{mtheorem: characterization of sigma complete}
    Assume that every irreducible ultrafilter is Dodd sound (See Definition \ref{Definition: special properties}(6)).
    Then a uniform $\sigma$-complete ultrafilter over a regular cardinal has the Galvin property if and only if it is a $D$-sum of $n$-fold sums of $\kappa$-complete $p$-points over $\kappa$, where $D$ is a $\sigma$-complete ultrafilter over $\lambda<\kappa$.
\end{mtheorem}
We note that in the above theorem, the ultrafilter $D$ might be just a $\sigma$-complete ultrafilter over a cardinals $\lambda<\kappa$ (see Theorem \ref{Theorem: D-limit of p-points}). By results of Schlutzenberg \cite{Schlutz}, in the Mitchell-Steel extender models $L[E]$, every irreducible ultrafilter is Dodd sound, so the assumption in the theorem holds in $L[E]$. Hence Theorem \ref{mtheorem: characterization of sigma complete} implies  that in the canonical inner models of the form of $L[E]$, even above a superstrong cardinal, the $n$-fold sum of $p$-points, in fact, \textit{characterizes} the  ultrafilters with the Galvin property. This characterization implies for example that $\sigma$-complete ultrafilters over successor cardinals always possess the Galvin property (Corollary \ref{Cor: ultrafilter on successor is Galvin}). 

As a corollary, we obtain the characterization of the Tukey-top ultrafilters:
\begin{corollary}[UA]
    Assume that every irreducible ultrafilter is Dodd sound, then a $\sigma$-complete ultrafilter over a regular cardinal is Tukey-top if and only if it is 
    not a $D$-sum of $n$-fold sums of $\kappa$-complete $p$-points over $\kappa$, where $D$ is a $\sigma$-complete ultrafilter over $\lambda<\kappa$. 
\end{corollary}
This corollary may come as a bit of a surprise
if one is familiar with the Tukey order on $\omega$: Dobrinen and Raghavan proved independently that it is consistent that there are non-Tukey-top
ultrafilters on \(\omega\) that are not $n$-fold sums of $p$-points \cite{Blass/Dobrinen/Raghavan15}, more specifically, a generic ultrafilter for $P(\omega\times \omega)/\text{fin}\cdot\text{fin}$ is such an ultrafilter; this result was stretched by Dobrinen in
\cite{DobrinenJSL15,DobrinenJML2016}.

One might suspect that under these very restrictive assumptions, we again run into the situation where every $\kappa$-complete ultrafilter has the Galvin property, but by theorem \ref{mthorem: non-Galvin implies existence of ultrafilter}, a non-Galvin cardinal suffices to guarantee the existence of a non-Galvin ultrafilter. Our next result suggests that in the canonical inner models, non-Galvin cardinals are exactly the large cardinal assumption needed to ensure the existence of non-Galvin ultrafilters:
\begin{mtheorem}[UA]\label{mtheorem: non-Galvin ulder UA is optimal}
    Assume that every irreducible ultrafilter is Dodd sound. If there is a $\kappa$-complete non-Galvin ultrafilter on an uncountable cardinal $\kappa$, then there is a non-Galvin cardinal.
\end{mtheorem}

One feature which seems to require more effort is to obtain a non-Galvin ultrafilter which extends the club filter (i.e. $q$-point). The ultrafilters that were constructed in \cite{TomNatasha} from a $\kappa$-compact cardinal extended the club filter and it is not clear at this point whether a non-Galvin cardinal implies the existence of such ultrafilters. Nonetheless, in the canonical inner models, the implication holds. In fact, the existence of a non-Galvin ultrafilter is equivalent to the existence of a non-Galvin $q$-point:
\begin{mtheorem}[UA]\label{mtheorem: extending the club filter under UA}
    Assume every irreducible ultrafilter
    is Dodd sound.
    Suppose $\kappa$ is 
    an uncountable cardinal
    that carries a $\kappa$-complete non-Galvin ultrafilter.
    Then the Ketonen least non-Galvin
    $\kappa$-complete
    ultrafilter on $\kappa$
    extends the closed unbounded filter.
\end{mtheorem}
The organization of this paper is as follows:
\begin{itemize}
    \item In section~\S1, we collect some basic definitions and facts from the theory of ultrafilters.
    \item In Section ~\S2, we establish the connection between non-Galvin ultrafilters and various  diamond-like principles. 
    \item In Section ~\S3, we use partial soundness to conclude that some ultrafilter is non-Galvin and define the corresponding diamond  $\diamondsuit^-_{\text{thin}}$.
    
    \item In Section ~\S4, we introduce the non-Galvin cardinals and prove Main Theorem \ref{mthorem: non-Galvin implies existence of ultrafilter}.
    
    \item In Section ~\S5, we work in the canonical inner models and prove Main Theorems \ref{mthorem: kappa complete galvin characterization},\ref{mtheorem: characterization of sigma complete},\ref{mtheorem: non-Galvin ulder UA is optimal},\ref{mtheorem: extending the club filter under UA}.
    
    \item  In Section ~\S6, we state some open questions and suggest further directions.
\end{itemize}
\subsection{Notation} Our notation is mostly standard. Let $\kappa$ be a cardinal and $X$ be any set. Then $[X]^\kappa=\{Y\in P(X)\mid |Y|=\kappa\}$ and $[X]^{<\kappa}=\{Y\in P(X)\mid |Y|<\kappa\}$. When $X$ is a set of ordinals, we identify elements of $[X]^{<\kappa}$ with their increasing enumerations. We write ${}^{<\kappa}X$ for the set of all functions $f\colon\gamma\rightarrow X$ where $\gamma<\kappa$ and ${}^\alpha X$ for the set of all functions $f\colon\alpha\rightarrow X$. 
Let $\kappa$ be regular. For two subsets of $\kappa$, we write  $X\subseteq^* Y$ to denote that  $X\setminus Y$ is bounded in $\kappa$. Similarly, for $f,g\colon\kappa\rightarrow\kappa$ we denote $f\leq^* g$ if there is $\alpha<\kappa$ such that for every $\alpha\leq\beta<\kappa$, $f(\beta)\leq g(\beta)$.  We say that $C\subseteq \kappa$ is a {\em closed unbounded} (or \textit{club}) subset of $\kappa$ if it is a closed subset with respect to the order topology on $\kappa$ and unbounded in the ordinals below $\kappa$. The {\em club filter} over \(\kappa\) is the filter:
$$\text{Club}_\kappa:=\{X\subseteq \kappa\mid X\text{ includes a closed unbounded subset of }\kappa\}.$$
If $f\colon A\rightarrow B$ is a function,
then $f``(X)=\{f(x)\mid x\in X\}$ and $f^{-1}[Y]=\{a\in A\mid f(a)\in Y\}$.

\section{Preliminaries}
  
 We only consider $\sigma$-complete ultrafilters over regular cardinals in this paper. We will, however, consider ultrafilters on \(\kappa\)
 that fail to be uniform or \(\kappa\)-complete. For a $\sigma$-complete ultrafilter $U$, we denote by $M_U$ the transitive collapse of the ultrapower of the universe of sets by $U$ and by $j_U\colon V\rightarrow M_U$ the usual ultrapower embedding. Given an elementary embedding $j\colon V\rightarrow M$ and an object $A\in M$, we let $\rho=\min\{\alpha\mid A\in V_{j(\alpha)}\}$ and define  $D(j,A):=\{X\subseteq V_\rho\mid A\in j(X)\}$. When $A$ is an ordinal, we will always replace $V_{\rho}$ in the above definition by $\rho$. If $M$ is any model of ZFC and \(f\) is a function or relation defined in the language of set theory, the relativization of \(f\) to this model is denoted by $(f)^M$; for example,
if \(\kappa\in M\), we might consider $(\kappa^+)^M$, $V_\kappa^M$, etc. 

The primary large cardinals we will be interested in are measurable cardinals. We say that
a cardinal $\kappa$ is \textit{measurable} if it  carries a non-principal $\kappa$-complete ultrafilter. In the introduction, we also mentioned the \textit{compact cardinals,} which can be characterized using the filter extension property: we say $\kappa$ has the  \textit{$\lambda$-filter extension property} if every $\kappa$-complete filter on $\lambda$ can be extended to a $\kappa$-complete ultrafilter. A \textit{$\kappa$-compact cardinal} is a cardinal $\kappa$ which has that $\kappa$-filter extension property. For more background on large cardinals, we refer the reader to \cite{kanamori1994}.

\begin{definition}[Special properties of ultrafilters]\label{Definition: special properties}
Let $U$ be an ultrafilter over a regular cardinal $\kappa$. We say that: \begin{enumerate}
\item A function \(f\) on \(\kappa\)
    is said to be \textit{constant (mod \(U\))} if there is a
    set \(A\in U\) such that \(f\restriction A\) is constant. A function $f$ is \textit{unbounded (mod $U$)} if $\forall \alpha<\kappa$, $f^{-1}[\alpha]\notin U$.
    A function \(f\) is \textit{almost one-to-one (mod \(U\))} if there is a set \(A\in U\)
    such that \(f\restriction A\) is almost one-to-one in the sense that for any \(x\),
    \(\{\alpha\in A : f(\alpha) = x\}\) is bounded below \(\kappa\).
    \item $U$ is a \textit{$p$-point} if every function $f\colon\kappa\rightarrow \kappa$ which is unbounded $(\text{mod }U)$ is almost one-to-one $(\text{mod }U)$.\footnote{Note that for $\kappa$-complete ultrafilters over $\kappa$ this is equivalent to the definition of $p$-points using the existence of pseudo-intersections \cite{Kanamori1976UltrafiltersOA}. In general, without assuming $\kappa$-completeness, these definitions are not equivalent.}
    \item $U$ is \textit{$\mu$-indecomposable} if for any function $f\colon \kappa\rightarrow \mu$, there is $\mu'<\mu$ such that $f^{-1}[\mu']\in U$.
    \item  $U$ is  \textit{weakly normal} if whenever $f\colon A\rightarrow \kappa$ is such that $A\in U$ and $f$ is regressive, there is $A'\subseteq A$, $A'\in U$ such that $f''[A']$ is bounded.\footnote{The notion of decomposability and weak normality makes sense also for filters when requiring the sets to be positive instead of measure $1$.}
    \item $U$ is \textit{$\alpha$-sound} if the function $j^\alpha\colon P(\kappa)\rightarrow M_U$ defined by $j^\alpha(X)=j_U(X)\cap \alpha$ belongs to $M_U$.
    \item  $U$ is \textit{Dodd sound} if it is $[\text{id}]_U$-sound.
    \item $U$ is \textit{$\kappa$-irreducible} if  every ultrafilter $W$ on an ordinal $\lambda<\kappa$ 
    that is Rudin-Frol\'ik below $U$ is principal.
    (See Definition \ref{def:RudinOrders}.)
\end{enumerate}
\end{definition}

\begin{remark}
    \begin{enumerate}
        \item The concept of Dodd soundness arose in inner model theory, where it
serves as a strong form of the initial segment condition \cite{Schimmerling1995}.
Though on first glance it may appear quite different, the 
Dodd soundness of a mouse is essentially equivalent to the Dodd soundness of its last extender as defined above. The formulation of Dodd soundness given here is due to Goldberg \cite{GoldbergUA}.
   \item  Note that if $U$ is $\alpha$-sound then $\{j_U(A)\cap\alpha\mid A\subseteq \kappa\}\in M_U$. This is in fact equivalent. Indeed, if $\{j_U(A)\cap\alpha\mid A\subseteq \kappa\}\in M_U$ then  it is the inverse of the transitive collapse of $\{j(S)\cap [\text{id}]_U\mid S\in P(\kappa)\}$. \item Note that if $U$ is an ultrafilter over a regular cardinal $\kappa$, and $\lambda<\kappa$ is such that $\lambda\in U$, then automatically, $U$ is a $p$-point as for any function $f:\kappa\rightarrow\kappa$, $f\restriction\lambda$ is bounded and hence there are no unbounded functions mod $U$.
    \item 
    If $U$ is irreducible and uniform on $\lambda$, then $U$ is $\lambda$-irreducible.
    \end{enumerate}
\end{remark}

\begin{proposition}\label{Proposition: Functions properties}
    Let $f\colon\kappa\rightarrow\kappa$ be any function and $U$ an ultrafilter over $\kappa$.
    \begin{enumerate}
    \item  $f$ is  unbounded mod $U$ if and only if $\sup_{\alpha<\kappa}j_U(\alpha)\leq[f]_U$.
        \item $f$ is almost one-to-one mod $U$ if and only if there is a (monotone) function $g\colon\kappa\rightarrow\kappa$ such that $j_U(g)([f]_U)=[g\circ f]_U\geq [\text{id}]_U$.
    \end{enumerate}
\end{proposition}
\begin{proof}
    $(1)$ is trivial. For $(2)$, Suppose that $f$ is almost one-to-one on $A\in U$, and let for each $\alpha<\kappa$ $g(\alpha)=\sup f^{-1}[\alpha+1]\cap A$. Then for each $\xi\in A$ $g(f(\xi))=\sup f^{-1}[f(\xi)+1]\cap A\geq \xi$, hence $[g\circ f]_U\geq [\text{id}]_U$. For the other direction, let $g$ be a monotone function such that $[g\circ f]_U\geq[\text{id}]_U$. Then there is a set $A\in U$ such that for each $\alpha\in A$, $g\circ f(\alpha)\geq\alpha$. Hence if $\beta\in f^{-1}[\alpha]$, then $g(\alpha)\geq g(f(\beta))\geq\beta$, hence $f^{-1}[\alpha]\subseteq g(\alpha)+1$.
\end{proof}
\begin{definition}\label{Definition: sums and limits}
Suppose $U$ is an ultrafilter over $X$ and for each $\alpha\in X$, $U_\alpha$ is an ultrafilter over $X_\alpha$. Define the limit
$$U\text{-}\lim \l U_\alpha\r_{\alpha\in X}=\big\{Y\subseteq X\mid \{\alpha\in X\mid Y\cap X_\alpha\in U_\alpha\}\in U\big\}$$
and the sum
$$\sum_U\l U_\alpha\r_{\alpha\in X}=\big\{Y\subseteq  \cup_{\alpha\in X}\{\alpha\}\times X_\alpha\mid \{\alpha\in X\mid (Y)_\alpha\in U_\alpha\}\in U\big\}$$
where $(Y)_\alpha=\{\beta\in X_\alpha\mid (\alpha,\beta)\in Y\}$ is the $\alpha^{\text{th}}$ fiber of $Y$.

The key property of sums is that they yield ultrafilters that
represent iterated ultrapowers:
\begin{lemma}[{\cite[Cor. 5.2.7]{GoldbergUA}}]
    If \(U\) is an ultrafilter on \(X\)
    and \(\langle W_\alpha\rangle_{\alpha\in X}\)
    is a sequence of ultrafilters, 
    then letting \(W^* = [\alpha\mapsto W_\alpha]_U\),
    \(M_{\sum_U \l W_\alpha\r_{\alpha\in X}} = (M_{W^*})^{M_U}\)
    and \(j_{\sum_U \l W_\alpha\r_{\alpha\in x}} = (j_{W^*})^{M_U}\circ j_U\).
    Moreover, \(U\text{-}\lim \l W_\alpha\r_{\alpha\in X} = j_U^{-1}[W^*].\)\qed
\end{lemma}
The sum construction is often used to obtain an ultrafilter
representing an iterated ultrapower in this way, and in this context, the
choice of the sequence \(\langle W_\alpha\rangle_{\alpha\in X}\)
representing \(W^*\) is usually irrelevant and distracting. For this reason, we introduce a notation that allows us to remain agnostic about this choice.

\begin{definition}\label{Definition:Frown}
    If \(U\) is an ultrafilter over $X$ and \(M_U\) satisfies that \(W^*\) is an ultrafilter, then \(U^\frown W^*\) denotes 
    \(\sum_U \l W_\alpha\r_{\alpha\in X}\), where
    \(W_\alpha\) is a sequence of ultrafilters such that \(W^* = [\alpha\mapsto W_\alpha]_U\).
\end{definition}

Technically, the
definition of \(U^\frown W^*\) depends on the choice of the
underlying sets of \(W_\alpha\). This 
ambiguity causes no issues, however, since if
\(W_\alpha'\) is another sequence such that 
\(W^* = [\alpha\mapsto W_\alpha']_U\),
then letting \(Z = \sum_U \l W_\alpha\r_{\alpha\in X}\) and \(Z' = \sum_U \l W'_\alpha\r_{\alpha\in X}\),
there is a set \(S\in Z\cap Z'\) such that \(Z \cap P(S) = Z'\cap P(S)\).

\begin{definition}
    We define recursively when $U$ is an \textit{$n$-fold sum of $p$-points}. $W$ is a $1$-fold sum of $p$-points if $W$ is a $p$-point. We say that $W$ is an $n+1$-fold sum of $p$-points if 
    there are $n$-fold sums of $p$-points $U_\alpha$ and a $p$-point ultrafilter $U$ such that $U$ is Rudin-Keisler equivalent to $\sum_{U}\l U_\alpha\r_{\alpha<\kappa}$. 
\end{definition}
We shall now prove a slight improvement of the form of ultrafilters which have the Galvin property in Theorem \ref{Theorem: Galvin Improvment}, this will be turn out to be an exact characterization of the ultrafilters with the Galvin property under UA plus every irreducible is Dodd sound in Main Theorem \ref{mtheorem: characterization of sigma complete}.  We need the definition of the modified diagonal intersection:
\begin{definition}
    Suppose that $W$ is a $\kappa$-complete ultrafilter over $\kappa$ and let $\pi_W:\kappa\rightarrow\kappa$ be the function which represents $\kappa$ mod $W$. For a sequence $\l A_i\r_{i<\kappa}$ of subsets of $\kappa$, we define the \textit{modified diagonal intersection} by
    $$\Delta^{W}_{i<\kappa}A_i=\{\alpha<\kappa\mid \forall i<\pi_W(\alpha), \ \alpha\in A_i\}$$
    
\end{definition}
\begin{fact}
If $W$ is a $\kappa$-complete ultrafilter over $\kappa$ and $\l A_i\r_{i<\kappa}\subseteq W$, then:
\begin{enumerate}
    \item $\Delta^W_{i<\kappa}A_i\in W$.
    \item for every $i_0<\kappa$, $(\Delta^W_{i<\kappa}A_i)\setminus (\pi^{-1}[i_0+1])\subseteq A_{i_0}$.
\end{enumerate}
\end{fact}
\begin{theorem}\label{Theorem: D-limit of p-points}
    Suppose that $\lambda<\kappa$, let $D$ be any ultrafilter over $\lambda$ and $\l W_\xi\r_{ \xi<\lambda}$ be a sequence of $n$-fold sums of $\kappa$-complete $p$-point ultrafilters over $\kappa$. Then $\sum_{D}\l W_\xi\r_{\xi<\lambda}$ has the Galvin property. 
\end{theorem}
\begin{proof}
Denote by $Z:=\sum_{D}\l W_\xi\r_{\xi<\lambda}$, and let us assume for simplicity of notation that $n=2$. Hence $Z=\sum_{D}\l\sum_{U_\xi}\l U_{\xi,\eta}\r_{\eta<\kappa}\r_{\xi<\lambda}$, where each $U_{\xi}$ and $U_{\xi,\eta}$ is a $\kappa$-complete $p$-point over $\kappa$. For $A\in Z$, define
    $$A^{(2)}_{i,j}=\{k<\kappa\mid \l i,j,k\r\in A\} $$
    $$A^{(1)}_i=\{j<\kappa\mid A^{(2)}_{i,j}\in U_{i,j}\}$$
    
    $$A^{(0)}=\{i<\lambda\mid A^{(1)}_i\in U_i\}.$$
     Note that $$A\in \sum_{D}\l\sum_{U_i}\l U_{i,j}\r_{j<\kappa}\r_{i<\lambda}\Leftrightarrow\{i<\lambda\mid (A)_i\in \sum_{U_i}\l U_{i,j}\r_{j<\kappa}\}\in D$$
     $$\Leftrightarrow\{ i<\lambda\mid \{j<\kappa\mid A^{(2)}_{i,j}\in U_{i,j}\}\in U_i\}\in D\Leftrightarrow A^{(0)}\in D.$$
     For any $W\in \{U_i\mid i<\lambda\}\cup\{U_{i,j}\mid i<\lambda,\ j<\kappa\}$, choose $\pi_W:\kappa\rightarrow\kappa$ such that $[\pi_W]_W=\kappa$ and $\pi_W$ is almost one-to-one. Such a function exists since $W$ is a $\kappa$-complete $p$-point. Define $\rho_W:\kappa\rightarrow\kappa$ by $$\rho_{W}(\alpha)=\sup\pi^{-1}_{W}[\alpha+1]+1.$$
     Next we define:  $$\rho^{(1)}(\alpha)=\sup_{i<\alpha}\rho_{U_i}(\alpha),\text{ and }\rho^{(2)}(\alpha)=\sup_{i,j<\alpha}\rho_{U_{i,j}}(\alpha).$$
     Note that $\rho^{(1)},\rho^{(2)}:\kappa\rightarrow\kappa$ since $\kappa$ is regular.
     Now we are ready to prove the Theorem. Let $\l A_i\r_{i<2^\kappa}$ be a sequence of sets in $Z$.
    Since $\lambda<\kappa$,we can assume without loss of generality that there is a set $A^{(0)}_*\in D$ such that for every $i<2^\kappa$, $A_*^{(0)}=(A_i)^{(0)}$.
    Let $\mathcal{N}$ be an elementary substructure of $H(\theta)$ for some large enough $\theta$ such that:
    \begin{enumerate}
        \item $|\mathcal{N}|=\kappa$.
        \item ${}^{<\kappa}\mathcal{N}\subseteq\mathcal{N}$.
        \item $\kappa\subseteq \mathcal{N}$ and $\kappa^+\cap\mathcal{N}\in \kappa^+$.
        \item  $\l A_i\r_{i<2^\kappa}\in \mathcal{N}$.
    \end{enumerate}
    Let $\alpha^*=\kappa^+\cap\mathcal{N}$.  
    \begin{claim}\label{claim: elementary submodel}
        For every $\l\alpha_1,\alpha_2\r\in [\kappa]^2$ and $\delta<\alpha^*$, there is $\delta<\beta<\alpha^*$ such that 
    \begin{enumerate}
        \item $\forall i\in (A_*)^{(0)}, (A_{\beta})^{(1)}_i\cap \alpha_1=(A_{\alpha^*})^{(1)}_i\cap \alpha_1$.
        \item $\forall i\in (A_*)^{(0)}\forall j<\alpha_1, \  (A_{\beta})^{(2)}_{i,j}\cap \alpha_2=(A_{\alpha^*})^{(2)}_{i,j}\cap\alpha_2$.
    \end{enumerate}
    \end{claim}
    \begin{proof}
        Consider the statement $$\phi(\alpha_1,\alpha_2,\delta)\equiv\exists \beta>\delta \ (1)\wedge(2)$$ $H(\theta)\models \phi(\alpha_1,\alpha_2,\delta)$ as witnessed by $\alpha^*$ and since $\alpha_1,\alpha_2,\delta\in \mathcal{N}$, the elementarity of $\mathcal{N}$ implies that there is such $\beta\in\mathcal{N}$ and in particular $\beta<\alpha^*$.   
    \end{proof}
    Define a sequence $\l \mu_i\mid i<\kappa\r$ inductively, suppose that $\l \mu_j\mid j<i\r$ was defined. Let $\delta=\sup_{j<i}\mu_j+1\in \mathcal{N}$ and apply the claim to $\delta$ and $$\  \alpha_1=\rho^{(1)}(i),\text{ and } \alpha_2=\rho^{(2)}(i)$$ to produce $\mu_i>\delta$ (and thus $\mu_i\neq \mu_j$ for all $j<i$). We claim that $$\bigcap_{i<\kappa}A_{\mu_i}\in \sum_D(\l\sum_{U_i}\l U_{i,j}\r_{j<\kappa}\r_{i<\lambda}.$$  

     To see this, we define for every $\xi\in (A_*)^{(0)}$,
$$(A_*)^{(1)}_\xi=(A_{\alpha^*})^{(1)}_\xi\cap\Delta^{U_\xi}_{i<\kappa}(A_{\mu_i})^{(1)}_{\xi}\setminus \rho_{U_\xi}(\xi)$$
     and for every $\xi\in (A_*)^{(0)}$, $\eta\in (A_*)^{(1)}_\xi$, define $$(A_*)^{(2)}_{\xi,\eta}=(A_{\alpha^*})^{(2)}_{\xi,\eta}\cap \Delta^{U_{\xi,\eta}}_{i<\kappa}(A_{\mu_i})^{(2)}_{\xi,\eta}\setminus \rho_{U_{\xi,\eta}}(\eta)$$
     Let $$A_*=\bigcup_{\xi\in A_*^{(0)}}\bigcup_{\eta\in (A_*)^{(1)}_\xi}\{\xi\}\times\{\eta\}\times (A_*)^{(2)}_{\xi,\eta}$$
     \begin{claim}
     For every $\l\alpha,\beta,\gamma\r\in A_*$, and for every $i<\kappa$, $\alpha\in (A_{\mu_i})^{(0)}$, $\beta\in (A_{\mu_i})^{(1)}_{\alpha}$ and $\gamma\in (A_{\mu_i})^{(2)}_{\alpha,\beta}$.
     \end{claim}
     \begin{proof}[Proof of claim.]
     Let $\l \alpha,\beta,\gamma\r\in A_*$. By definition of $A_*$, $\alpha\in (A_*)^{(0)}$, $\beta\in (A_*)^{(1)}_\alpha$ and $\gamma\in (A_*)^{(2)}_{\alpha,\beta}$. In particular, $$(*)\ \ \ \alpha<\pi_{U_\alpha}(\beta)\text{ and 
 }\beta<\pi_{U_{\alpha,\beta}}(\gamma).$$
     For $i<\kappa$, we note first that $\alpha\in (A_{\mu_i})^{(0)}$ since we assume $(A_{\mu_i})^{(0)}=(A_*)^{(0)}$. 
     Now to see that $\beta\in (A_{\mu_i})^{(1)}_\alpha$, split into cases. If $i<\pi_{U_\alpha}(\beta)$, then $\beta\in (A_{\mu_i})^{(1)}_\alpha$  by the definition of the modified diagonal intersection.
     If $i\geq \pi_{U_{\alpha}}(\beta)$, then  $\beta<\rho_{U_\alpha}(i)$. Also, by $(*)$, $\alpha<\pi_{U_\alpha}(\beta)\leq i$ and therefore $\rho_{U_\alpha}(i)\leq\sup_{\alpha<i}\rho_{U_\alpha}(i)=\rho^{(1)}(i)$. By the choice of $\mu_i$, (1) of Claim \ref{claim: elementary submodel}
     $$\beta\in (A_{\alpha^*})^{(1)}_\alpha\cap \rho^{(1)}(i)=(A_{\mu_i})^{(1)}_\alpha\cap \rho^{(1)}(i).$$
     Finally for $\gamma$, if $i<\pi_{U_{\alpha,\beta}}(\gamma)$, then $\gamma\in (A_{\mu_i})^{(2)}_{\alpha,\beta}$. If $i\geq \pi_{U_{\alpha,\beta}}(\gamma)$, then as in the previous paragraph, $\beta<\pi_{U_{\alpha,\beta}}(\gamma)\leq i$ and thus $$\gamma<\rho_{U_{\alpha,\beta}}(i)\leq \rho^{(2)}(i).$$ We conclude that $\gamma\in (A_{\alpha^*})^{(2)}_{\alpha,\beta}\cap \rho^{(2)}(i)$. By the choice of $\mu_i$ and (2) of Claim \ref{claim: elementary submodel},  $\gamma\in (A_{\mu_i})^{(2)}_{\alpha,\beta}\cap \rho^{(2)}(i)$. 
     \end{proof}
     By the claim, that for every $\l\alpha,\beta,\gamma\r\in A_*$ and every $i<\kappa$,  $\l \alpha,\beta,\gamma\r\in A_{\mu_i}$, namely
$A_*\subseteq\bigcap_{i<\kappa}A_{\mu_i}$. Finally, we note that $A_*\in Z$. Indeed,  $(A_*)^{(0)}\in D$ by the choice of $(A_*)^{(0)}$. Also, for every $i<\kappa$, and $\alpha\in (A_*)^{(0)}$, $\alpha\in (A_{\mu_i})^{(0)}$ and so $(A_{\mu_i})^{(1)}_\alpha\in U_\alpha$. We conclude $(A_*)^{(1)}_\alpha\in U_\alpha$. Also, for  $\beta\in (A_*)^{(1)}_{\alpha}$, $\beta\in (A_{\mu_i})^{(1)}_{\alpha}$ and therefore $(A_{\mu_i})^{(2)}_{\alpha,\beta}\in U_{\alpha,\beta}$. It follows that $(A_*)^{(2)}_{\alpha,\beta}\in U_{\alpha,\beta}$. Hence $A_*\in Z$, and in particular $\bigcap_{i<\kappa}A_{\mu_i}\in Z$.
\end{proof}

 \end{definition}
 Recall that the sequence of $\l U_\alpha\r_{\alpha\in X}$ is called \textit{discrete} if there is a sequence of pairwise disjoint sets $\l A_\alpha\r_{\alpha\in X}$  such that $A_\alpha\in U_\alpha$. We say that $\l U_\alpha\r_{\alpha\in X}$ is discrete mod $U$, if there is $Y\in U$, $Y\subseteq X$ such $\l U_\alpha\r_{\alpha\in Y}$ is discrete. 
 \begin{fact}     $\sum_U\l U_\alpha\r_{\alpha<\kappa}\equiv_{RK} U\text{-}\lim\l U_\alpha\r_{\alpha<\kappa}$ iff $\l U_\alpha\r_{ \alpha<\kappa}$ is discrete mod $U$.
 \end{fact}
\begin{proposition}\label{descretep-ppoints}
     If $U$ is a $p$-point ultrafilter, then any sequence $\l U_\alpha\r_{ \alpha<\kappa}$ of distinct $\kappa$-complete  ultrafilters is discrete  mod $U$.
 \end{proposition}
 \begin{proof}
     See \cite[Cor. 5.15]{Kanamori1976UltrafiltersOA}.
 \end{proof}
\begin{definition}(Orderings of ultrafilters)\label{def:RudinOrders}
Let $U,W$ be ultrafilters over ordinals $\kappa,\lambda$ (resp.) define:
\begin{enumerate}
    \item the \textit{Rudin-Keisler order} by $U\leq_{RK}W$ if there is a function $\pi\colon \lambda\rightarrow \kappa$ such that $U=\{B\subseteq \kappa\mid \pi^{-1}[B]\in W\}$.
    \item the \textit{Rudin-Frol\'ik} order by $U\leq_{RF}W$ if there is a set $I\in U$ and a discrete sequence $\l W_i\r_{i\in I}$ of ultrafilters over $\kappa$ such that $W=U\text{-}\lim\l W_i\r_{i\in I}$.
    \item the \textit{Ketonen order} by $U<_{\Bbbk}W$ if $j_W''U$ is contained in a countably complete ultrafilter $U^*$ of \(M_W\) such that $[id]_W\in U^*$. 
\end{enumerate}
\end{definition}

For more background on ultrafilters, their orderings, and the Ultrapower Axiom, we refer the reader to \cite{Kanamori1976UltrafiltersOA} and \cite{GoldbergUA}.

 We also record here the definition and basic properties of the \textit{canonical functions}.
\begin{definition}\label{Definition: Canonical functions}
For every $\eta<\kappa^+$, we fix a cofinal sequence $\l \eta_i\r_{i<\cf(\eta)}$. Define recursively the canonical functions $f_\alpha\colon\kappa\rightarrow\kappa$ for $\alpha<\kappa^+$ as follows:
$f_0=0$ is the constant function with value $0$.
Given $f_\alpha$, define $f_{\alpha+1}(x)=f_\alpha(x)+1$.
For limit $\eta<\kappa^+$ we split into cases:
\begin{enumerate}
    \item if $\cf(\eta)<\kappa$, define $f_\eta(x)=\sup_{i<\cf(\eta)}f_{\eta_i}(x)$.
    \item if $\cf(\eta)=\kappa$, define $f_\eta(x)=\sup_{i<x}f_{\eta_i}(x)$.
\end{enumerate}
\end{definition}
It is not hard to see that the canonical functions are increasing modulo the bounded ideal, but
the main reason we are interested in those functions is the following: \begin{proposition}\label{Prop: Canonical functions}
Let $k\colon N\rightarrow M$ be an elementary embedding (not necessarily definable in $N$) with critical point $\kappa$. Then for every $\alpha<(\kappa^+)^N$, $k(f_\alpha)(\kappa)=\alpha$. 
\end{proposition}
\begin{proof}
    By induction on $\alpha$. Clearly, for $\alpha=0$, $k(f_0)(\kappa)=0$ and if $k(f_\alpha)(\kappa)=\alpha$ then by elementarity $k(f_{\alpha+1})(\kappa)=\alpha+1$.
    For limit $\eta$, if $\cf(\eta)<\kappa$, then the functions used in the definition of $f_\eta$ are $\l f_{\eta_i}\r_{i<\cf(\eta)}$ are pointwise mapped by $k$;
    that is, $k(\l f_{\eta_i}\mid i<\cf(\eta)\r)=\l k(f_{\eta_i})\mid i<\cf(\eta)\r$. It follows by elementarity and the definition of $f_\eta$ that
    $k(f_\eta)(\kappa)=\sup_{i<\cf(\eta)}k(f_{\eta_i})(\kappa)$.
    Hence by the induction hypothesis, $k(f_\eta)(\kappa)=\sup_{i<\cf(\eta)}\eta_i=\eta$.
    If $\cf(\eta)=\kappa$ then the sequence $\l f_{\eta_i}\mid i<\kappa\r$ is stretched by $k$ to  $k(\l f_{\eta_i}\mid i<\kappa\r)=\l f'_{\eta_i}\mid i<k(\kappa)\r$ but for every $i<\kappa$,
    as $k(i)=i$, we have $f'_{\eta_i}=k(f_{\eta_i})$. Again by the definition of $f_\eta$, elementarity, and the induction hypothesis, we conclude that:
    $$k(f_\eta)(\kappa)=\sup_{i<\kappa}f'_{\eta_i}(\kappa)=\sup_{i<\kappa}k(f_{\eta_i})(\kappa)=\sup_{i<\kappa}\eta_i=\eta.$$
\end{proof}
\section{Diamond-like principle and the Galvin property}
In \cite{bgs}, a relation between Kurepa trees and the Galvin property has been established to construct a $\kappa$-complete non-Galvin ultrafilter. In this section, we exploit the deep connection between Kurepa trees and diamond principles which was first observed by Jensen \cite{JensenCombinatorics},  to find new combinatorial properties of ultrafilters which ensures the Galvin property.  
\begin{definition}\label{Definition: Diamond star}Let $S$ be a stationary set. $\diamondsuit^*(S)$ is the assertion that there is a sequence $\l\mathcal{A}_\alpha\r_{\alpha\in S}$ such that $\mathcal{A}_\alpha\subseteq P(\alpha)$ and:
\begin{enumerate}
    \item $|\mathcal{A}_\alpha|\leq\alpha$.
    \item for every $X\subseteq \kappa$ there is a club $C$ such that for each $\alpha\in C\cap S$, $C\cap\alpha,X\cap\alpha\in\mathcal{A}_\alpha$.
\end{enumerate}
\end{definition}
\begin{proposition}
    If $\diamondsuit^*(S)$ holds then any  ultrafilter $U$ over a regular cardinal $\kappa$ satisfying $\text{Club}_\kappa\cup\{S\}\subseteq U$ and $\cf^{M_U}([\text{id}]_U)\leq \crit(j_U)$ must be non-Galvin.
\end{proposition}
\begin{proof}
     Suppose otherwise, and let $C_X$ for every $X\subseteq\kappa$ be the club guaranteed by item $(2)$ of  $\diamondsuit^*(S)$. Then $C_X\in U$. Also, for each $\alpha\in S$, let $\l I^\alpha_i\r_{i<\cf(\alpha)}$ be a partition of $\mathcal{A}_\alpha$ such that $|I^\alpha_i|<\alpha$. Now for each $X\subseteq\kappa$, consider the function $f_X:C_{X}\cap S\rightarrow \kappa$ defined by $f_X(\alpha)=i<\cf(\alpha)$  for the unique $i$ such that $X\cap \alpha\in I^\alpha_i$. Since $cf^{M_U}([id]_U)\leq \crit(j_U)$, there is a function $\pi:\kappa\rightarrow On$ such that $i<\cf(\alpha)\leq\pi(\alpha)$ and $[\pi]_U=\crit(j_U)$. It follows that there is $A_X\subseteq C_{X}\cap S$, $A_X\in U$ and $\gamma_X<\kappa$ such that for every $\alpha\in A_X$, $f_X(\alpha)=\gamma_X$. There are $2^\kappa$-many subsets with the same $\gamma_X=\gamma^*$.  Now apply Galvin's property to those $2^\kappa$-many sets in order find $\kappa$-many distinct subsets of $\kappa$, $\l X_\xi\r_{\xi<\kappa}$ for which $A^*:=\bigcap_{\xi<\kappa}A_{X_\xi}\in U$. Now for each $\alpha\in A^*\cap S$, $|I^\alpha_{\gamma^*}|<\alpha$. Since $\kappa$ is regular, we may apply F\"{o}dor's lemma to find a stationary set $S'\subseteq A^*\cap S$  and $\theta<\kappa$ such that $|I^\alpha_{\gamma^*}|=\theta$ for each $\alpha\in S'$. Consider $\l X_i\r_{i<\theta^{+}}$ and for each $i\neq j<\theta^{+}$ let $\beta_{i,j}<\kappa$ be high enough so that $X_i\cap \beta_{i,j}\neq X_j\cap \beta_{i,j}$. Take any $\alpha\in S'\setminus\sup_{i\neq j<\theta^{+}}\beta_{i,j}$. To reach a contradiction, note that on one hand, since $\alpha\in S'$, $|I^{\alpha}_{\gamma^*}|=\theta$. On the other hand, for every $i\neq j<\theta^{+}$, $X_i\cap \alpha\in I^{\alpha}_{\gamma^*}$ and the sets $X_i\cap\alpha$ are all distinct.  
 \end{proof}
Let us introduce a similar guessing principle $\diamondsuit^*_{\text{thin}}(U)$ to the one above, which can be formulated in terms of the ultrapower and does not involve the club filter. Then we will prove that $\diamondsuit^*_{\text{thin}}(U)$ implies that $U$ is non-Galvin.  
\begin{definition}\label{Definition: Diamond-Thin}
    
An ultrafilter \(W\) 
on a regular cardinal \(\kappa\) satisfies
\(\diamondsuit^*_\text{thin}(W)\)
if there is a sequence
of sets \(\langle\mathcal A_\alpha\rangle_{\alpha< \kappa}\) such that:
\begin{enumerate}
    \item for all \(A\subseteq \kappa\),
for \(W\)-almost all \(\alpha\),
\(A\cap \alpha \in \mathcal A_\alpha\).
\item \(\alpha\mapsto |A_\alpha|\) is not almost one-to-one mod \(W\).

\end{enumerate}
The sequence $\l\mathcal{A}_\alpha\r_{\alpha<\kappa}$ is called a \textit{$\diamondsuit^*_{\text{thin}}(U)$-sequence}.
\end{definition}
In the ultrapower, this is expressed as follows:
\begin{lemma}\label{Lemma: Ultrapower rep. of Diamond*}
    $\diamondsuit^*_{\text{thin}}(U)$ is equivalent to the existence of a set $A\in M_U$ such that:
    \begin{enumerate}
        \item $\{j_U(S)\cap [id]_U\mid S\subseteq\kappa\}\subseteq A$.
        \item  there is no function $f\colon\kappa\rightarrow\kappa$ such that $j_U(f)(|A|^M)\geq [id]_U$.
    \end{enumerate}
\end{lemma}
\begin{proof}
    The witnessing $\diamondsuit^*_{\text{thin}}(U)$-sequence is just the sequence  $\l \mathcal{A}_\alpha\r_{\alpha<\kappa}$ representing $A$ in $M_U$. Clearly, condition $(1)$ is equivalent to the fact that for every $S\subseteq \kappa$, $\{\alpha<\kappa\mid S\cap\alpha\in\mathcal{A}_\alpha\}\in U$. By Proposition \ref{Proposition: Functions properties}, condition $(2)$ is equivalent to the function $\alpha\mapsto |\mathcal{A}_\alpha|$ not being almost one-to-one mod $U$.
\end{proof}

\begin{lemma}\label{Lemma: Thin implies non-Galvin}
    If \(\diamondsuit^*_\textnormal{thin}(W)\),
    then \(W\) is non-Galvin.
\end{lemma}
\begin{proof}
    Assume towards  contradiction that $W$ has the Galvin property. Enumerate $\mathcal{A}_\alpha=\{A_{\alpha,i}\mid i<|\mathcal{A}_\alpha|\}$.  For every set $X$, there is $B_X\in W$ such that for every for every $\alpha\in B_X$, $X\cap\alpha\in \mathcal{A}_\alpha$. 
    By our assumption, there are $\kappa$-many distinct sets $\{X_i\mid i<\kappa\}$ such that $B:=\bigcap_{i<\kappa}B_{X_i}\in W$. Note that the key property of $B$ is that for every $i<\kappa$ and for all $\alpha\in B$, $X_i\cap\alpha\in\mathcal{A}_\alpha$. Since the function $\alpha\mapsto|\mathcal{A}_\alpha|$ is not almost one-to-one mod $W$, there  is $\theta<\kappa$ and an unbounded subset $B'\subseteq B$ such that for every $\alpha\in B'$, $|\mathcal{A}_\alpha|=\theta$.  Consider  $\{ X_i\mid i<\theta^+\}$. For every $i\neq j<\theta^+$, find $\alpha_{i,j}<\kappa$ such that $X_i\cap \alpha_{i,j}\neq X_j\cap\alpha_{i,j}$ and take $\alpha^*=\sup_{i,j<\theta^+}\alpha_{i,j}$. By regularity of $\kappa$, $\alpha^*<\kappa$. Since $B'$ is unbounded there exists some $\beta^*\in B'$ with $\beta^*>\alpha^*$. It follows that for every $i<\theta^+$, $X_i\cap\beta^*\in \mathcal{A}_{\beta^*}$, and also for every $i\neq j$, since $\alpha_{i,j}<\beta^*$, $X_i\cap\beta^*\neq X_j\cap\beta^*$. It follows that $i\mapsto X_i\cap \beta^*$ is a one-to-one function from $\theta^+$ into $\mathcal{A}_{\beta^*}$. This contradicts the fact that $\beta^*\in B'$ and thus $|\mathcal{A}_{\beta^*}|=\theta$.
\end{proof}

 \begin{corollary}
Suppose that $\kappa$ is regular and $U$ is an ultrafilter extending the club filter on $\kappa$. Assume that there is a sequence of sets $\l \mathcal{A}_\alpha\r_{ \alpha<\kappa}$ such that:
 \begin{enumerate}
        \item for every $\alpha<\kappa$, $|\mathcal{A}_\alpha|<\alpha$.
        \item for every $X\subseteq \kappa$, $\{\alpha<\kappa\mid X\cap\alpha\in \mathcal{A}_\alpha\}\in U$.
    \end{enumerate}
    Then $\diamondsuit^*_{\text{thin}}(U)$ holds and in particular $U$ is non-Galvin.
\end{corollary}
\begin{proof}
It remains to show that $\alpha\mapsto |\mathcal{A}_\alpha|$ is not one-to-one on a set in $U$. If $A\in U$, then $A$ is stationary since $\text{Club}_\kappa\subseteq U$. By F\"{o}dor applied to the function $\alpha\mapsto |\mathcal{A}_\alpha|$ restricted to $A$, there is an unbounded subset $S'\subseteq A$ and $\theta<\kappa$ such that for every  $\alpha\in S'$, $|\mathcal{A}_\alpha|=\theta$. In particular, $\alpha\mapsto |\mathcal{A}_\alpha|$ is not almost one-to-one on $A$.\end{proof}
The most important class of ultrafilters which satisfy $\diamondsuit^*_{\text{thin}}$ are the non $p$-point Dodd sound ultrafilters as will be proven in Lemma \ref{EquivDoddSound}. To prove that lemma, we will need the following characterization due to Goldberg of Dodd sound ultrafilters \cite[Thm. 4.3.26]{GoldbergUA}:
\begin{theorem}\label{Theorem: Characterization of Dodd sound}
    A uniform ultrafilter $U$ on an ordinal $\delta$ is Dodd sound if and only if there is a sequence $\l \mathcal{A}_\alpha\r_{\alpha<\delta}$  such that for any sequence
$\l S_\alpha\subseteq\alpha \r_{\alpha < \delta}$, the following are equivalent:
\begin{enumerate}
    \item [(a)] There is a set $S\subseteq\kappa$ such that for $U$-almost every $\alpha$, $S\cap\alpha=S_\alpha$.
    \item [(b)] For $U$-almost every $\alpha$, $S_\alpha\in \mathcal{A}_\alpha$.
\end{enumerate}
\end{theorem}
\begin{lemma}\label{EquivDoddSound}
    Let $\kappa$ be regular and $U$ a non $p$-point Dodd sound ultrafilter, then $\diamondsuit^*_{\text{thin}}(U)$.  
\end{lemma}
\begin{proof}
    Assume that $U$ is a non $p$-point Dodd sound ultrafilter. Let $\l\mathcal{A}_\alpha\r_{\alpha<\kappa}$ be the sequence obtained by Theorem \ref{Theorem: Characterization of Dodd sound}.
    Note that for every $S\subseteq \kappa$, the sequence $\l S\cap\alpha\r_{\alpha<\kappa}$ satisfies condition $(a)$ of Theorem \ref{Theorem: Characterization of Dodd sound}. 
    By the theorem, we conclude that for $U$-almost every $\alpha$, $S\cap\alpha\in \mathcal{A}_\alpha$. It follows that $j_U(S)\cap [id]_U\in [\alpha\mapsto \mathcal{A}_\alpha]$ and $\{j_U(S)\cap [id]_U\mid S\subseteq\kappa\}\subseteq [\alpha\mapsto \mathcal{A}_\alpha]_U$. 
    Similarly, from the implication  $(b)$ to $(a)$ we deduce that $[\alpha\mapsto \mathcal{A}_\alpha]_U\subseteq\{j_U(S)\cap[\text{id}]_U\mid S\subseteq \kappa\}$. By Dodd soundness, the function $j^{[\text{id}]_U}\colon P(\kappa)\rightarrow \{j_U (S)\cap[\text{id}]_U\mid S\subseteq\kappa\}$ defined by $j^{[\text{id}]_U}(S)=j(S)\cap [\text{id}]_U$ belongs to $M_U$.  Thus $M_U\models |[\alpha\mapsto \mathcal{A}_\alpha]_U|=2^\kappa$. Finally, $\alpha\mapsto |\mathcal{A}_\alpha|$ cannot be an almost one-to-one function mod $U$: otherwise, the class of any unbounded function $\kappa\leq [\pi]_U$ would also be an almost one-to-one mod $U$. To see this, suppose that $[\tau]_U=\kappa$, then $[\alpha\mapsto 2^{\tau(\alpha)}]_U=2^\kappa$ and by our assumption, this is represented by an almost one-to-one function mod $U$ \footnote{Being an almost one-to-one function mod $U$ is clearly a property of an equivalence class mod $U$.}. Let $X\in U$ be the set witnessing that $\alpha\mapsto 2^{\tau(\alpha)}$ is almost one-to-one mod $U$. Also we let $Y\in U$ be such that for every $\alpha\in Y$, $\tau(\alpha)\leq\pi(\alpha)$. We claim that $\pi\restriction X\cap Y$ is almost one-to-one as for any $\gamma<\kappa$, $$\{\alpha<\kappa\mid \pi(\alpha)<\gamma\}\cap X\cap Y\subseteq\{\alpha<\kappa\mid \tau(\alpha)<\gamma\}\cap X\cap Y\subseteq$$
    $$\subseteq\{\alpha<\kappa\mid 2^{\tau(\alpha)}\leq 2^\gamma\}\cap X\cap Y.$$ The right most set is bounded by the choice of $X$. We conclude that $U$ is a $p$-point contradiction.
\end{proof}
Note that an ultrafilter $U$ satisfying $\diamondsuit^*_{\text{thin}}(U)$ need not be Dodd sound since by Lemma \ref{Lemma: Ultrapower rep. of Diamond*} we only cover the set $\{j_U(S)\cap[id]_U\mid S\subseteq \kappa\}$. However, at least for $\kappa$-complete Dodd sound ultrafilters, the second requirement of $\diamondsuit^*_{\text{thin}}(U)$ regarding the function $\alpha\mapsto |\mathcal{A}_\alpha|$ is equivalent to $U$ not being a $p$-point.
\begin{proposition}\label{Proposition: diamond if and only if Dodd sound}
     Let $\kappa$ be measurable and $U$ be a  $\kappa$-complete Dodd sound ultrafilter over $\kappa$, and let $[\alpha\mapsto \mathcal{A}_\alpha]_U=\{j_U(S)\cap [\text{id}]_U\mid S\subseteq \kappa\}$.  Then $U$ is a non $p$-point ultrafilter if and only if the function $\alpha\mapsto |\mathcal{A}_\alpha|$ is not almost one-to-one mod $U$.
\end{proposition}
\begin{proof} One direction follows from the previous lemma. Let us prove the other, note that $\alpha\mapsto|\mathcal{A}_\alpha|$ cannot be bounded on a set in $U$, just otherwise, suppose that $\theta<\kappa$ is such that $B^*:=\{\alpha<\kappa\mid |\mathcal{A}_\alpha|\leq\theta\}\in U$. Take any $\theta^+$-many sets $\{ X_i\mid i<\theta^+\}$ such that there is $\gamma<\kappa$ such that for all $i\neq j<\theta^+$, $X_i\cap\gamma\neq X_j\cap\gamma$. For each $i<\theta^+$, Denote by $B_i:=\{\alpha<\kappa\mid X_i\cap\alpha\in\mathcal{A}_\alpha\}\in U$. By $\kappa$-completeness and fineness, there is $\gamma^*\in B^*\cap(\bigcap_{i<\theta^+}B_i)\setminus\gamma$. It follows that $|\mathcal{A}_{\gamma^*}|=\theta$ but also for each $i<\theta^+$, $X_i\cap \gamma^*\in \mathcal{A}_{\gamma^*}$ are all distinct sets. Contradiction. We conclude that $\alpha\mapsto |\mathcal{A}_\alpha|$ is an unbounded function mod $U$ which is also not almost one-to-one according to $(1)$. Hence $U$ is not a $p$-point.
\end{proof}
We cannot drop the $\kappa$-completeness assumption here:
\begin{example}
    Suppose that $W$ is a fine normal ultrafilter over $P_\kappa(\lambda)$ for $\kappa<\lambda$ where $\lambda$ is a regular cardinal. By \cite[Theorems 4.4.37 \& 4.4.25]{GoldbergUA}, there is a Dodd sound non uniform ultrafilter $U$ on $\lambda$ (and therefore $p$-point) which is Rudin-Keisler equivalent to  $W$. Note that there is no function which is unbounded (and therefore no function which is almost one-to-one) mod $U$. In particular, $\alpha\mapsto |\mathcal{A}_\alpha|$ is not almost one-to-one mod $U$. Also, note that $U$ satisfies $\diamondsuit^*_{\text{thin}}(U)$ and therefore is an example of  a non-Galvin ultrafilter  over $\lambda$ which is uniform and not $\lambda$-complete.  
\end{example}

\begin{corollary}\label{Corollary: nonppoint+DoddSound implies non-Galvin}
    If $U$ is a non $p$-point, Dodd sound ultrafilter over a regular cardinal $\kappa$, then $U$ is non-Galvin.
\end{corollary}
In attempt to pinpoint the exact guessing principle that catches non-Galvinness, we note that the usage of $\diamondsuit^*_{\text{thin}}(W)$ in the argument of Lemma \ref{Lemma: Thin implies non-Galvin} can be replaced with the following weakening:
\begin{definition}\label{Definition: W-Kurepa} Let $\kappa\leq\lambda\leq 2^\kappa$. An ultrafilter \(W\) 
on a regular cardinal \(\kappa\) satisfies
\(\diamondsuit^*_\text{par}(W,\lambda)\)
if there is a sequence
of sets \(\langle X_\alpha\rangle_{\alpha<\lambda}\), $A\in M_W$ such that:
\begin{enumerate}
    \item $\{j_U(X_\alpha)\cap [id]_U\mid \alpha<\lambda\}\subseteq A$.
\item For any function $f:\kappa\rightarrow\kappa$, $j_U(f)(|A|^{M_W})<[id]_W$.

\end{enumerate}
\end{definition}
Clearly, $\diamondsuit^*_{\text{thin}}(W)$ implies $\diamondsuit^*_{\text{par}}(W,2^\kappa)$ which in turn imply $\diamondsuit^*_{\text{par}}(W,\lambda)$ for any $\lambda\in [\kappa,2^\kappa]$.
\begin{proposition}
    $\diamondsuit^*_{\text{par}}(W,\lambda)$ implies that $\neg \text{Gal}(W,\kappa,\lambda)$
\end{proposition}
\begin{proof}
    The argument of Lemma \ref{Lemma: Thin implies non-Galvin}  gives this stronger result.
\end{proof}

The principle $\diamondsuit^*_{\text{par}}(W,\lambda)$ is equivalent to the existence of a set $K\subseteq P(\kappa)$ of size $\lambda$ and a sequence $\l A_\alpha\r_{\alpha<\kappa}$ such that:
\begin{enumerate}
    \item For every $X\in K$, $\{\alpha<\kappa\mid X\cap\alpha\in A_\alpha\}\in W$.
    \item  The function $\alpha\mapsto |A_\alpha|$ is not almost one-to-one mod $W$.
\end{enumerate}
The referee pointed out to us the strong similarity of  $\diamondsuit^*_{\text{par}}$ to the notion of pseudo-Kurepa families due to Todorcevic \cite{todorcevic1991}. Indeed, many of the initial segments of the sets in $K$ must be equal in order for the sets $A_\alpha$ of asymptotically bounded cardinality to exist.  



Next, we would like to provide two closure properties of the class of ultrafilters satisfying $\diamondsuit^*_{\text{thin}}$.
\begin{lemma}
    Suppose $U$ is an ultrafilter on \(\kappa\) and
    $Z$ is the $U$-limit of a discrete
    sequence of
     ultrafilters
    $W_\xi$
    on \(\kappa\) such that 
    \(\diamondsuit^*_\textnormal{thin}(W_\xi)\).
    Then 
    \(\diamondsuit^*_\textnormal{thin}(Z)\).
    
\end{lemma}
    \begin{proof}
        Fix a partition of \(\kappa\)
        into sets \(S_\xi\in W_\xi\).
        For each \(\xi < \kappa\),
        let \(\langle\mathcal A^\xi_\alpha\rangle_{\alpha < \kappa}\)
        witness that 
        \(\diamondsuit^*_\textnormal{thin}(W_\xi)\). Then 
        let 
        \(\mathcal A_\alpha = 
        \mathcal A^\xi_\alpha\)
        where \(\xi < \kappa\)
        is unique such that \(\alpha\in S_\xi\). Fixing $A\subseteq \kappa$, we would like to show that $B:=\{\alpha<\kappa\mid A\cap\alpha\in \mathcal{A}_{\alpha}\}\in U\text{-}\lim \l W_\xi\r_{\xi<\kappa}$. For any $\xi<\kappa$, then $B_\xi:=\{\alpha\in S_\xi\mid A\cap\alpha\in \mathcal{A}^\xi_\alpha\}\in W_\xi$. Since for each $\alpha\in S_\xi$, $\mathcal{A}_\alpha=\mathcal{A}^\xi_\alpha$, we conclude that $B_\xi\subseteq B$ and therefore $B\in W_\xi$. It follows that $B\in U\text{-}\lim \l W_\xi\r_{\xi<\kappa}$. It remains to show that $c(\alpha)=|\mathcal{A}_\alpha|$ is not almost one-to-one on any set $B\in W$. Suppose otherwise, and let $B\in W$ witness that $c$ is almost one-to-one.  Pick any $\xi<\kappa$ such that $B\in W_\xi$ to reach a contradiction note that $B\cap S_\xi\in W_\xi$, and the function $c$ is almost one-to-one on this set. However, for every $\alpha\in B\cap S_\xi$, $\mathcal{A}^\xi_\alpha=\mathcal{A}_\alpha=c(\alpha)$ and so $\alpha\mapsto |\mathcal{A}^\xi_\alpha|$ is almost one-to-one on $B\cap S_\xi$, contradicting $\diamondsuit^*_{\text{thin}}(W_\xi)$.  
    \end{proof}

\begin{lemma}\label{Lemma: SumofDiamond}
    Suppose $U$ is an $n$-fold sum of $p$-points on \(\kappa\) and
    $\langle W_\xi\rangle_{\xi < \kappa}$ is a sequence of (not necessarily discrete)
    \(\kappa\)-complete ultrafilters 
    on \(\kappa\) such that 
    \(\diamondsuit^*_\textnormal{thin}(W_\xi)\).
    Then letting
    $Z = U\text{-}\lim\l W_\xi\r_{\xi < \kappa}$, we have \(\diamondsuit^*_\textnormal{thin}(Z)\).
    \begin{proof}
    We first consider the case that
    $U$ is a $p$-point. Then replace \(U\) with \(U_W = D(j_U, W)\)
    where $W$ is the point in $M_U$ represented by 
    $\xi\mapsto W_\xi$. Note that $U_W$ is Rudin-Keisler below an ultrafilter on $\kappa$ which implies that $U_W$ concentrates on a set of ($\kappa$-complete) ultrafilters of size $\kappa$. By enumerating those ultrafilters  $W'_\xi$ for $\xi<\kappa$, we can shift  $U_W$ to an ultrafilter $U'$ on $\kappa$ such that $[\text{id}]_{U_W}$ is identified with $[\xi\mapsto W'_\xi]_{U'}$. 
    Also, note that $U'-\text{lim} W'_\xi=U-\text{lim} W_\xi$ since the factor map   
    $k\colon M_{U'}\rightarrow M_U$ sends $k([\xi\mapsto W'_\xi]_{U'})=W$ and thus $$X\in U'-\lim\l W'_\xi\r_{\xi<\kappa}\Leftrightarrow j_{U'}(X)\in [\xi\mapsto W'_\xi]_{U'}\Leftrightarrow $$
    $$\Leftrightarrow j_U(X)=k(j_{U'}(X))\in W\Leftrightarrow X\in U-\lim \l W_\xi\r_{\xi<\kappa}.$$ 
    
    Since \(U'\leq_{RK}U\), and $U$ is a $p$-point,  $U'$ is also a $p$-point (see \cite[Cor 2.8]{Kanamori1976UltrafiltersOA}). The sequence $\l W'_\xi\r_{ \xi<\kappa}$  represents the identity in $U'$, it is one-to-one mod $U'$, since all the $W'_\xi$'s are $\kappa$-complete, by Proposition \ref{descretep-ppoints} the sequence is  discrete on a set in $U'$.\footnote{Note that even if the $W_\xi$'s we started with were not distinct, the $W_\xi'$'s will be distinct on a set in $U'$. For example, if $W_\xi=W_0$ for every $\xi$, then $U_W$ is the principle ultrafilter concentrating on $\{W_0\}$ and thus $U'$ is principle and $W_0=W_\xi'$. It is still true that on a measure one set in $U'$, i.e. $\{0\}$, the sequence $\l W'_\xi\r_{\xi<\kappa}$ is distinct. In this case, the lemma is trivial as $Z=W_0$.} This allows us to apply the previous lemma,
    obtaining thin diamond for 
    $U'\text{-}\lim \l W'_\xi\r_{\xi<\kappa} = U\text{-}\lim\l W_\xi\r_{\xi < \kappa}$.
    
    Now suppose the lemma is true for
    $n$-fold sums of $p$-points,
    and we will prove it when $U$ is
    an $n+1$-fold sum. 
    We can fix a $p$-point $D$
    such that $U$ is the 
    $D$-limit of a sequence of $n$-fold sum $p$-points $U_\xi$ on $\kappa$. As in the previous paragraph, since $D$ is a $p$-point, we may assume that the $U_\xi$'s are discrete.
    Let $U^* = [\xi\mapsto U_\xi]_D$, then by elementarity, $M_D\models$ $U^*$ is an $n$-fold sum of $p$-points. Applying the induction hypothesis in
    $M_D$ to $U^*$ and the ultrafilters $j_D(\langle W_\xi\rangle_{\xi < \kappa})=\l Z^*_\xi\r_{ \xi<j_D(\kappa)}$, we conclude that $Z^*=U^*\text{-}\lim\l Z^*_\xi\r_{\xi<j_D(\kappa)}$ satisfies $\diamondsuit^*_{\text{thin}}(Z^*)$. Let $[\alpha\mapsto Z_\alpha]_D=Z^*$ and assume without loss of generality that for every $\alpha<\kappa$, $\diamondsuit^*_{\text{thin}}(Z_\alpha)$ holds. We claim that
    \[(*) \ \ Z= D\text{-}\lim\l Z_\alpha\r_{\alpha < \kappa} = U\text{-}\lim\l W_\xi\r_{\xi < \kappa}\]
    from which it follows that $\diamondsuit^*_\text{thin}(Z)$,
    by the argument of the previous paragraph. To see $(*)$, since we assumed that the $U_\alpha$'s are discrete, by the theory of sums and limits of ultrapower $$j_{\sum_D\l U_\alpha\r_{\alpha<\kappa}}=j_{D\text{-}\lim\l U_\alpha\r_{\alpha<\kappa}}=j_{U^*}\circ j_D\text{ and }[\text{id}]_{D\text{-}\lim\l U_\alpha\r_{\alpha<\kappa}}=[\text{id}]_{U^*},$$ hence $$X\in D\text{-}\lim\l Z_\alpha\r_{\alpha<\kappa}\Leftrightarrow j_D(X)\in Z^*=U^*\text{-}\lim\l Z^*_\xi\r_{\xi<j_D(\kappa)}\Leftrightarrow$$
    $$\Leftrightarrow j_{U^*}(j_D(X))\in j_{U^*}(j_D(\l W_\xi\r_{\xi<\kappa}))([\text{id}]_{U^*})\Leftrightarrow$$
    $$\Leftrightarrow j_{D\text{-}\lim\l U_\alpha\r_{\alpha<\kappa}}(X)\in j_{D\text{-}\lim\l U_\alpha\r_{\alpha<\kappa}}(\l W_\xi\r_{\xi<\kappa})([\text{id}]_{D\text{-}\lim\l U_\alpha\r_{\alpha<\kappa}})\Leftrightarrow$$
    $$\Leftrightarrow X\in (D\text{-}\lim\l U_\alpha\r_{\alpha<\kappa})\text{-}\lim\l W_\xi\r_{\xi<\kappa}\Leftrightarrow X\in U\text{-}\lim\l W_\xi\r_{\xi<\kappa}.$$
    
    \end{proof}
\end{lemma}

\section{Partial Dodd soundness and skies}
A finer analysis of the diamond-like principles of the previous section reveals that partial soundness suffices for an ultrafilter to be non-Galvin. To better understand this improvement, let us prove the following theorem in terms of general elementary embeddings.
\begin{theorem}\label{Theorem: Partial sound}
    Suppose that $j\colon V\rightarrow M$ is an elementary embedding with $\crit(j)=\kappa$ such that $\lambda=\sup\{j(f)(\kappa)\mid f\colon \kappa\rightarrow\kappa\}$ and $\{j(A)\cap \lambda\mid A\subseteq\kappa\}\in M$. Then there is $\xi$ such that $D:=D(j,\xi)$ and $\neg \Gal{D,\kappa,2^\kappa}$.
\end{theorem}
\begin{remark}
    Note that from the assumptions of the Theorem it follows that $\lambda<j(\kappa)$, indeed, if $\lambda=j(\kappa)$, then since we are assuming $\{j(A)\cap\lambda\mid A\subseteq\kappa\}\in M$, we have $j''P(\kappa)\in M$ and therefore $\{j(f)(\kappa)\mid f:\kappa\rightarrow\kappa\}\in M$. It follows that $M\models cf(j(\kappa))=2^\kappa$. But by elementarity, $M\models j(\kappa)$ is regular. Contradiction.
\end{remark}
\begin{proof}
    Denote $\mathcal{A}=\{j(A)\cap\lambda\mid A\subseteq\kappa\}\in M$. Enumerate $V_\kappa$ in $V$, $f\colon \kappa\rightarrow V_\kappa$ such that for every $x\in V_\kappa$, $f^{-1}[x]$ is unbounded in $\kappa$. Since $\mathcal{A}\in (V_{j(\kappa)})^M$, there is $j(\kappa)>\xi\geq\lambda$ such that $j(f)(\xi)=\mathcal{A}$.
    By similar arguments we can ensure that there are some functions $g, h: \kappa \rightarrow \kappa$ such that for the same $\xi$ we will also have $\kappa=j(g)(\xi)$ and $\lambda=j(h)(\xi)$.
    Let $D=D(j,\xi)$, $j_D\colon V\rightarrow M_D$ be the ultrapower and $k_D\colon M_D\rightarrow M$ be the factor map $k_D([\phi]_D)=j(\phi)(\xi)$. Note that $$\lambda=k_D([h]_D), \  \kappa=k_D([g]_D), \ \mathcal{A}=k_D([f]_D)$$ and therefore $\kappa,\lambda,\mathcal{A}\in Im(k_D)$. It follows that $\crit(k_D)>\kappa$ and $[g]_D=\kappa$. Since $$k_D([h]_D)=\lambda\leq\xi=k_D([id]_D),$$ the elementarity of $k_D$ implies that $[h]_D\leq [id]_D$. 
    Recall that for any function $\phi\colon \kappa\rightarrow \kappa$, $j(\phi)(\kappa)<\lambda$ thus by elementarity of $k_D$, 
    $$(*)\text{ for any function }\phi\colon \kappa\rightarrow \kappa, \ j_D(\phi)(\kappa)<[h]_D.$$ By our initial assumption, $\lambda>j(\alpha\mapsto 2^\alpha)(\kappa)=2^\kappa$ and since $M\models |\mathcal{A}|=2^\kappa$, $$M_D\models |[f]_D|=2^{[g]_D}<[h]_D.$$ Denote by $B_\alpha=f(\alpha)$, 
    note that  and fix a set $X^*\in D$ such that if $\alpha\in X^*$ then $|B_\alpha|=2^{g(\alpha)}<h(\alpha)$. 
    Pick any $2^\kappa$ distinct subsets of $\kappa$, $\l A_\alpha\r_{ \alpha<2^\kappa}$, then $j(A_\alpha)\cap \lambda\in \mathcal{A}$ and by elementarity $j_D(A_\alpha)\cap \lambda'\in B$. It follows that
    $$X_\alpha:=\{\xi<\kappa\mid A_\alpha\cap h(\xi)\in B_\xi\}\in D$$
    We claim that $\l X_\alpha\r_{ \alpha<2^\kappa}$ witness that $\neg \text{Gal}(U,\kappa,2^\kappa)$.
    Otherwise, there is $I\in[2^\kappa]^\kappa$ such that $X_I:=\cap_{i\in I}X_i\in D$. Let us argue that there must be $\theta<\kappa$ such that 
    $$\sup\{h(\xi):\xi\in X_I\cap X^*,\ 2^{g(\xi)}<\theta\}=\kappa.$$ To see this, assume otherwise, then for each $\theta<\kappa$ we can define $$\rho(\theta)=\sup\{h(\xi)\mid \xi\in X_I\cap X^*, \  2^{g(\xi)}\leq 2^\theta\}$$ then $\rho\colon\kappa\rightarrow \kappa$ is well defined. Since $2^{j_D(g)([id]_D)}=2^{[g]_D}=2^\kappa$, we conclude that $j_D(\rho)(\kappa)\geq j_D(h)([\text{id}]_D)=[h]_D$, contradicting $(*)$. We proceed as before, find  $\beta\in X_I\cap X^*$ such that the restriction of $\theta$-many of the sets in $I$ to $h(\beta)$ are distinct. This produces a contradiction.
\end{proof}
 Let us define the concept of a \textit{sky of an elementary embedding at $\delta$}, which was first considered in the case that $\delta=\omega$ by Puritz \cite{Puritz1971,Puritz1972} and generalized to measurable cardinals later by Kanamori \cite{Kanamori1976UltrafiltersOA}. This concept will enable us to simplify our future definitions.
 \begin{definition}
     Let $j:V\rightarrow M$ be an elementary embedding where $M$ is transitive and $\kappa$ be any cardinal. We define a transitive relation on $[\sup (j''{\kappa}),j(\kappa))$: 
     $\alpha\preceq \beta$ if there is a function $f:\kappa\rightarrow\kappa$ such that $j(f)(\beta)\geq\alpha$. We derive the equivalence relation $\alpha\equiv \beta$ if $\alpha\preceq \beta$ and $\beta\preceq \alpha$. A \textit{sky of $j$ at $\kappa$} is a $\equiv$-equivalence class. We denote by $sky(\alpha)$ the sky of $\alpha$ at $\kappa$ for the unique $\kappa$ such that $\alpha\in[\sup (j''{\kappa}),j(\kappa))$. 
 \end{definition}
Note that the only interesting situation is when $\kappa$ is not a continuity point of $j$. Since $M$ is transitive, $\prec$ is a well-defined well-ordering of the skies. Moreover, since $\alpha\leq \beta$ implies $\alpha\preceq\beta$ then each sky is a half-open interval. 

Suppse now that $U$ is a $\sigma$-complete ultrafilter over $\kappa$. It is clear that for any $\alpha<j_U(\kappa)$, $\alpha\preceq [id]_U$ as $\alpha=[f]_U$ for some $f:\kappa\rightarrow\kappa$ and therefore $\alpha=j_U(f)([id]_U)$. So  $sky([id]_U)$ is the maximal sky.
This simple observation, together with Proposition \ref{Proposition: Functions properties}, leads to  an elegant characterization of $p$-points in terms of skies:
\begin{corollary}
    Let $U$ be a $\sigma$-complete ultrafilter over $\kappa$, then $U$ is a $p$-point if and only if $j_U$ has a unique sky at $\kappa$. 
\end{corollary}
We can now reformulate Theorem \ref{Theorem: Partial sound} in terms of skies:
 \begin{corollary}
    Suppose that $U$ is a $\kappa$-complete, $\lambda$-sound ultrafilter over $\kappa$ such that $\lambda$ is the least element of the second sky at $\kappa$. Then $U$ is non-Galvin. 
\end{corollary}
\begin{proof}
    By the definition of $\xi$ in the proof of Theorem \ref{Theorem: Partial sound}, we can choose $\xi=[\text{id}]_U$ and the theorem ensures that $U=D(j_U,[\text{id}]_U)$ is non-Galvin.
\end{proof}
Note that a embedding $j$ with critical point $\kappa$ has at least two skies at $\kappa$ if and only if $\sup\{j(f)(\kappa)\mid f\colon \kappa\rightarrow\kappa\}<j(\kappa)$.
\begin{corollary}\label{Cor: Superstrong}
    Suppose that there is a superstrong embedding $j\colon V\rightarrow M$ with $\crit(j)=\kappa$ and at least two skies. Then $\kappa$ carries a non-Galvin ultrafilter.
\end{corollary}

The reason that $\diamondsuit^*_{\text{thin}}$ (Definition \ref{Definition: Diamond star}) is not equivalent to  Dodd soundness is that we are only trying to cover $\{j_U(S)\cap [id]_U\mid S\subseteq\kappa\}$ with a set $A$ in $M_U$, while in Dodd soundness we need the actual set $\{j_U(S)\cap[id]_U\mid S\subseteq\kappa\}$ to be in $M_U$. Let us call this property \textit{covering soundness}. The innovation here is to work with covering $\lambda$-soundness which is just the ability to cover $\{j_U(S)\cap\lambda\mid S\subseteq\kappa\}$.

However, without any further assumptions, we can always take $P^{M_U}(\lambda)$ as our covering set, so covering $\lambda$-soundness is always true. What makes $\diamondsuit^*_{\text{thin}}$ non-trivial is the second requirement  that there is no function $f:\kappa\rightarrow\kappa$ such that $j_U(f)(|A|^{M_U})\geq [id]_U$. This rules out our previous example of $P^{M_U}(\lambda)$ or any other trivial example. Equipped with our new terminology of skies, we note that $(2)$ of Definition \ref{Definition: Diamond star} is in fact equivalent to $|A|^M$ not laying the top sky (namely $sky(|A|^{M_U})\prec sky([id]_U)$). 

Assuming (full) $\lambda$-soundness, the results of this section ensure that the ``covering" set could be chosen to be precisely $\{j_U(S)\cap\lambda\mid S\subseteq \kappa\}$. Moreover, with this choice, the $M_U$-cardinality of the covering set is $2^\kappa$. Then, under the assumption on $\lambda$ in Theorem \ref{Theorem: Partial sound} there is no function $f:\kappa\rightarrow\kappa$ such that $j_U(f)(2^\kappa)\geq \lambda$. 

Bearing the idea of skies in mind, we see the following common theme: if $A$ is the covering set and $\lambda$ is the degree of covering soundness then $sky(|A|^{M_U})\prec sky(\lambda)$. Let us formulate a diamond-like principle which generalize both Theorem \ref{Theorem: Partial sound} and $\diamondsuit^*_{\text{thin}}(U)$. It corresponds to covering $\lambda$-soundness, allowing $\lambda$ to lay in an arbitrary sky (except the least one). This diamond-like principle is essential to  prove the characterization of $\sigma$-complete non-Galvin ultrafilters.
\begin{definition}
Let $U$ be an ultrafilter over a regular cardinal $\kappa$.  $\diamondsuit^-_{\text{thin}}(U)$ is the statement that there is $A\in M_U$ and $\lambda<j_U(\kappa)$ such that:
    \begin{enumerate}
        \item $\{j_U(S)\cap\lambda\mid S\subseteq\kappa\}\subseteq A$. 
        \item there is no function $f\colon\kappa\rightarrow \kappa$ such that $j_U(f)(|A|^M)\geq \lambda$\footnote{ i.e. $sky(|A|^M)\prec sky(\lambda)$.}.
    \end{enumerate}
    Clearly $\diamondsuit^*_{\text{thin}}(U)$ implies $\diamondsuit^-_{\text{thin}}(U)$ by taking $\lambda=[id]_U$. Also, 
    \begin{corollary}\label{corollary: lambda sound and diamond minus}
        If $U$ is an ultrafilter over a regular cardinal $\kappa$ which is $\lambda$-sound where  $\lambda$ is such that for every function $f:\kappa\rightarrow\kappa$, $j_U(f)(\kappa)<\lambda$, then $\diamondsuit^-_{\text{thin}}(U)$.
    \end{corollary}
    \begin{proof}
        By $\lambda$-soundness of $U$, $A:=\{j_U(S)\cap\lambda\mid S\subseteq \kappa\}\in M_U$ and $M_U\models |A|=2^\kappa$. There cannot be a function $g:\kappa\rightarrow\kappa$ such that $j_U(g)(2^\kappa)\geq \lambda$, since otherwise, the function $g'(\alpha)=g(2^\alpha)$ would be a function from $\kappa$ to $\kappa$ such that $j_U(g')(\kappa)\geq \lambda$, contradicting the assumptions of the corollary. 
    \end{proof}
    \begin{theorem}\label{Theorem: diamond minus implies non-Galvin}
        $\diamondsuit^-_{\text{thin}}(U)$ implies that $U$ is non-Galvin.
    \end{theorem}
\end{definition}
\begin{proof}
    Fix any $\l X_\alpha\r_{\alpha<2^\kappa}$ sequence of distinct subsets of $\kappa$. $[\alpha\mapsto A_\alpha]_U=A$ and $[f]_U=\lambda=j_U(f)([id]_U)$. By our assumption,
    $$B_\alpha=\{\xi<\kappa\mid X_\alpha\cap f(\xi)\in A_\alpha\}\in U$$
    We claim that $\l B_\alpha\r_{ \alpha<2^\kappa}$ witness that $\neg \text{Gal}(U,\kappa,2^\kappa)$.
    Otherwise, there is $I\in[2^\kappa]^\kappa$ such that $B_I:=\cap_{i\in I}B_i\in U$. Consider the map $\xi\mapsto |A_\xi|$, note that  $|A_\xi|\leq \pi(\xi)$ where $j_U(\pi)([\text{id}]_U)=|A|$, and therefore there must be $\theta<\kappa$ such that 
    $$\sup\{f(\xi):\xi\in B_I,\ \pi(\xi)<\theta\}=\kappa.$$ Just assume otherwise, then for each $\theta<\kappa$ we can define $$g(\theta)=\sup\{f(\xi)\mid \xi\in B_I, \  \pi(\xi)\leq\theta\}$$ then $g\colon\kappa\rightarrow \kappa$ is well defined. Since $j_U(\pi)([\text{id}]_D)=|A|$ we conclude that $j_U(g)(|A|)\geq j_U(f)([\text{id}]_D)=\lambda$, contradicting condition $(2)$. Now the continuation is as before, we find  $\beta\in B_I$ such that $f(\beta)$ is high enough so that the restriction of $\theta^{+}$-many of the sets in $I$ to $f(\beta)$ are distinct. This produces a contradiction.
\end{proof}
The advantage of using the class of ultrafilters satisfying $\diamondsuit^-_{\text{thin}}(U)$ over the class satisfying $\diamondsuit^*_{\text{thin}}$, is that is it upward closed  with respect to the Rudin-Keisler ordering.
\begin{lemma}\label{Lemma: Rudin-Keisler upword closure of diamond minus}
    Suppose that $\diamondsuit^-_{\text{thin}}(U)$ holds and $U\leq_{RK}W$, then $\diamondsuit^-_{\text{thin}}(W)$ holds.
\end{lemma}
\begin{proof}
    Let $k:M_U\rightarrow M_W$ be an elementary embedding such that $j_W=k\circ j_U$ and $A,\lambda$ witnessing $\diamondsuit^-_{\text{thin}}(U)$. For every $S\subseteq \kappa$, we have $$j_W(S)\cap k(\lambda)=k(j_U(S)\cap\lambda)\in k(A).$$ Hence $\{j_W(S)\cap k(\lambda)\mid S\subseteq \kappa\}\subseteq k(A)\in M_W$. By elementarity, $|k(A)|^{M_W}=k(|A|^{M_U})$. Suppose toward contradiction that there is a function $g:\kappa\rightarrow\kappa$ such that $j_W(g)(k(|A|^{M_U}))\geq k(\lambda)$, then $k(j_U(g)(|A|))\geq k(\lambda)$ and by elementarity if $k$, $j_U(g)(|A|)\geq \lambda$, contradiction. 
\end{proof}
\begin{lemma}\label{Lemma: discrete limit of diamond minus}
    Suppose that
    $Z$ is an ultrafilter on \(\kappa\) which is the $U$-limit of a discrete
    sequence of
     ultrafilters
    $W_\xi$ on $\kappa$ and
     such that 
    \(\diamondsuit^-_\textnormal{thin}(W_\xi)\).
    Then 
    \(\diamondsuit^-_\textnormal{thin}(Z)\).
    
\end{lemma}
    \begin{proof}
        Fix a partition of \(\kappa\)
        into sets \(S_\xi\in W_\xi\).
        For each \(\xi < \kappa\),
        let \(\langle\mathcal A^\xi_\alpha\rangle_{\alpha < \kappa}\) and $f_\xi$
        witness that 
        \(\diamondsuit^-_\textnormal{thin}(W_\xi)\). Then 
        let 
        \(\mathcal A_\alpha = 
        \mathcal A^\xi_\alpha\)
        where \(\xi < \kappa\)
        is unique such that \(\alpha\in S_\xi\) and $f(\alpha)=f_\xi(\alpha)$. Let $A\subseteq \kappa$, we would like to show that $B:=\{\alpha<\kappa\mid A\cap f(\alpha)\in \mathcal{A}_{\alpha}\}\in U\text{-}\lim \l W_\xi\r_{\xi<\kappa}$. Take any $\xi<\kappa$, then $B_\xi:=\{\alpha\in S_\xi\mid A\cap f_\xi(\alpha)\in \mathcal{A}^\xi_\alpha\}\in W_\xi$. Since for each $\alpha\in S_\xi$ and $f(\alpha)=f_\xi(\alpha)$, $\mathcal{A}_\alpha=\mathcal{A}^\xi_\alpha$, we conclude that $B_\xi\subseteq B$ and therefore $B\in W_\xi$. It follows that $B\in U\text{-}\lim\l W_\xi\r_{\xi<\kappa}$. It remains to show that $c(\alpha)=|\mathcal{A}_\alpha|$ is in a lower sky than $f$. Suppose otherwise and let $g:\kappa\rightarrow\kappa$ such that for some $B\in W$, $\alpha\in B\rightarrow g(c(\alpha))\geq f(\alpha)$.  Pick any $\xi<\kappa$ such that $B\in W_\xi$ to reach a contradiction note that $B\cap S_\xi\in W_\xi$, and for every $\alpha\in B\cap S_\xi$, $g(|A^\xi|_\alpha)=g(c(\alpha))\geq f(\alpha)=f_\xi(\alpha)$. However, the sky $\alpha\mapsto |\mathcal{A}^\xi_\alpha|$ is below the sky of $f_\xi$, contradicting the choice of $f_\xi$.  
    \end{proof}
    For a non-discrete sequence, we have the following:
    \begin{lemma}\label{Lemma: non discrete}
        Suppose that $Z$ is an ultrafilter over $\kappa$ which is Rudin-Keisler equivalent to $\sum_U\l W_\xi\r_{\xi<\lambda}$, where $U$ is any ultrafilter over $\lambda\leq\kappa$ and $W_\xi's$ are ultrafilters over $\kappa$ such that $\diamondsuit^-_{\text{thin}}(W_\xi)$ holds. Then $\diamondsuit^-_{\text{thin}}(Z)$ holds. 
    \end{lemma}
    \begin{proof}
        Let $W^*=[\xi\mapsto W_\xi]_U$. By our assumption,  $$M_U\models W^*\text{ is an ultrafilter over }j_U(\kappa)\text{ and } \diamondsuit^-_{\text{thin}}(W^*).$$
        Let $j_{W^*}:M_{U}\rightarrow M_{W^*}$ be the ultrapower of $M_U$ by $W^*$. It follows that there is $A\in M_{W^*}$ and $\lambda<j_{W^*}(j_U(\kappa))$ such that $\{j_{W^*}(S)\cap\lambda\mid S\in P(j_U(\kappa))^{M_U}\}\subseteq A$
        and there is no function $f:j_U(\kappa)\rightarrow j_U(\kappa)\in M_U$ such that $j_{W^*}(f)(|A|^{M_{W^*}})\geq\lambda$. Note that $M_{W^*}=M_{\sum_U\l W_\xi\r_{\xi<\lambda}}$ and $j_{\sum_U\l W_\xi\r_{\xi<\lambda}}=j_{W^*}\circ j_U$. We claim that $A$ and $\lambda$ witness 
        that $\diamondsuit^-_{\text{thin}}(\sum_U\l W_\xi\r_{\xi<\lambda})$. Indeed, for any $X\subseteq \kappa$, $j_U(X)\in P(j_U(\kappa))^{M_U}$ and therefore $j_{W^*}(j_U(X))\cap \lambda\in A$. 
        Similarly, for any function $f:\kappa\rightarrow \kappa$, $j_U(f):j_U(\kappa)\rightarrow j_U(\kappa)\in M_U$ and therefore $j_{W^*}(j_U(f))(|A|^{M_{W^*}})<\lambda$. 
    \end{proof}
\section{Non-Galvin cardinals}
As pointed out in the introduction, a measurable cardinal does not imply the existence of a non-Galvin ultrafilter \cite{Parttwo}. In \cite{SatInCan}, the question regarding which large cardinal properties imply the existence of non-Galvin ultrafilters was raised and in \cite{TomNatasha} a $\kappa$-compact cardinal was proven to carry such an ultrafilter. We open this section with a new large cardinal property:
\begin{definition}\label{Definition: non-Galvin cardinal}
    $\kappa$ is called \textit{non-Galvin cardinal} if there are elementary embeddings $j\colon V\rightarrow M$, $i\colon V\rightarrow N$, $k\colon N\rightarrow M$ such that:
    \begin{enumerate}
        \item $k\circ i=j$.
        \item $\crit(j)=\kappa$, $\crit(k)=i(\kappa)$.
        \item ${}^{\kappa}N\subseteq N$ and ${}^{\kappa}M\subseteq M$
        \item there is $A\in M$ such that $i''\kappa^+\subseteq A$ and $M\models|A|<i(\kappa)$.
    \end{enumerate}
\end{definition}
Note that by condition $(4)$, $\kappa\subseteq A$ and that $A$ can be chosen so that $\min(A\setminus \kappa)=i(\kappa)$. 

The next proposition implies that we may assume that the embedding $j$ in the definition of non-Galvin cardinals is an ultrapower embedding and the embedding $i$ is an extender ultrapower derived from it.

\begin{proposition}\label{prp:nongalvin_via_uf}
    Suppose that $j\colon V\rightarrow M$, $i\colon V\rightarrow N$, $k\colon N\rightarrow M$ and $A\in M$ are as in Definition \ref{Definition: non-Galvin cardinal}. Then there is a $\kappa$-complete ultrafilter $U$ over $V_\kappa$ and $\rho<j_U(\kappa)$ which, together with the ultrapower by the  $(\kappa,\rho)$-extender $E$ derived from $j_U$ and $[\text{id}]_U$, witnesses that $\kappa$ is non-Galvin. Namely, the following hold:
    \begin{enumerate}
        \item $k_{E}\circ j_{E}=j_U$.
        \item  $\crit(j_U)=\kappa$, $\crit(k_{E})=\rho=j_{E}(\kappa)$.
        \item ${}^\kappa M_{E}\subseteq M_{E}$ and ${}^\kappa M_U\subseteq M_U$.
        \item $j_{E}''\kappa^+\subseteq [\text{id}]_U$ and $M_U\models |[\text{id}]_U|<j_{E}(\kappa)$.
    \end{enumerate}
\end{proposition}
\begin{proof}
    We may assume \(\sup(A) = \sup i''\kappa^+\) and \(A\cap i(\kappa) = \kappa\).
    Let \(U\) be the ultrafilter derived from \(j\) using \(A\).
    Let \(\bar A = [\text{id}]_U\)
    and let \(k_U : M_U\to M\) be the unique elementary embedding with \(k_U\circ j_U = j\) and \(k_U(\bar A) = A\). Note that \(\kappa\)
    and \(i(\kappa)\) are in the range of \(k_U\) since these ordinals are definable in \(M\) using \(A\)
    as a parameter: \(\kappa\) is the least ordinal not in \(A\), and
    \(i(\kappa) = |\sup(A)|^M\).
    Therefore \(k_U(\kappa) = \kappa\). Let \(\rho\) be such that \(k_U(\rho) = i(\kappa)\).
    
    Let \(E\) be the extender of length \(\rho\) derived from \(j_U\). Let \(k_E : M_E\to M_U\) denote the unique factor embedding with
    \(k_E \circ j_E = j_U\) and
    \(k_E\restriction \rho = \text{id}.\)

\[\begin{tikzcd}
	& {M_E} & {M_U} \\
	V \\
	& N & M
	\arrow["j_E", from=2-1, to=1-2]
	\arrow["{k_E}", from=1-2, to=1-3]
	\arrow["i"', from=2-1, to=3-2]
	\arrow["k"', from=3-2, to=3-3]
	\arrow["j"', bend right=60, from=2-1, to=3-3]
	\arrow["{j_U}", bend right=-60, from=2-1, to=1-3]
	\arrow["{k_U}", from=1-3, to=3-3]
\end{tikzcd}\]

    We will verify (1), (2), (3), and (4). Of course, (1) is true essentially by the definition of \(k_E\).
    
    For (2), note that \(\crit(j_U) = \kappa\) since \(k_U\circ j_U = j\) and \(k_U(\kappa) = \kappa\). 
    The fact that \(\crit(k_E) \geq \rho\) follows from the definition of \(k_E\). To see \(\crit(k_E) = \rho\) and \(k_E(\rho) = j_U(\kappa)\), we will show that\footnote{For a model $M$, $Hull^M(A)$ denotes the usual closure of the class $A\subseteq M$ under the Skolem functions of $M$, which in the case of an ultrapower simplifies to $Hull^{M_U}(A)=\{j_U(f)(\xi)\mid f:\kappa\rightarrow V, \xi\in A\}$.} \[\text{Hull}^{M_U}(j_U''V\cup \rho)\cap j_U(\kappa) = \rho\]
    This will establish that \(\crit(k_E) =\rho\) and \(k_E(\rho) = j_U(\kappa)\) since \(k_E\) is the
    inverse of the transitive collapse of \(\text{Hull}^{M_U}(j_U''V\cup \rho)\). 
    To prove this equality, it suffices to show the inclusion
    \(\text{Hull}^{M_U}(j_U''V\cup \rho)\cap j_U(\kappa)\subseteq \rho\).

    Since \(k\circ i=j\) and since $\crit(k)=i(\kappa)$, we have
    \(\text{Hull}^M(j''V\cup i(\kappa))\subseteq k''N\), and
    therefore \(\text{Hull}^M(j''V\cup i(\kappa))\cap j(\kappa) \subseteq k''N\cap j(\kappa) = i(\kappa)\). Since $$k_U''[\text{Hull}^{M_U}(j_U''V\cup \rho)]\subseteq \text{Hull}^M(j''V\cup i(\kappa)),$$ we have
    \[\text{Hull}^{M_U}(j_U''V\cup \rho) \subseteq k_U^{-1}[\text{Hull}^M(j''V\cup i(\kappa))]\] In particular,
    \begin{align*}\text{Hull}^{M_U}(j_U''V\cup \rho)\cap j_U(\kappa) &\subseteq k_U^{-1}[\text{Hull}^M(j''V\cup i(\kappa))\cap j(\kappa)] \\
    &= k_U^{-1}(i(\kappa)) = \rho\end{align*}

    Since \(k_E(\rho) = j_U(\kappa) > \rho\) and \(k_E \circ j_E = j_U\),
    it follows 
    \(\rho = j_E(\kappa)\). Note also that \(\rho < j_U(\kappa)\), and so the fact that \(k_E(\rho) = j_U(\kappa)\) implies
    \(k_E(\rho) \neq \rho\)
    and hence \(\crit(k_E) = \rho\).

    For (3), the inner model \(M_U\) is closed under \(\kappa\)-sequences since it is the ultrapower of \(V\) by a \(\kappa\)-complete ultrafilter. The inner model \(M_E\) is closed under \(\kappa\)-sequences by Lemma \ref{lma:extender_closure} below, since 
    \(\cf(\rho) > \kappa\) and \({}^\kappa\rho\subseteq M_E\). To see that \(\cf(\rho) > \kappa\), note that \(\cf(i(\kappa)) > \kappa\) since \(N\) satisfies that \(i(\kappa)\) is regular and \(N\) is closed under \(\kappa\)-sequences. Therefore \(M\) satisfies that \(\cf(i(\kappa)) > \kappa\).
    By the elementarity of \(k_U\), and since \(k_U(\rho) = i(\kappa)\), \(M_U\) satisfies \(\cf(\rho) > \kappa\). Here we use that \(k_U(\kappa) = \kappa\).

    Finally, we verify (4). By elementarity of \(k_U\), since \(|A| < i(\kappa)\), we have \(|\bar A| < \rho\). So we just have to show that
    \(j_E''\kappa^+\subseteq \bar A.\)
    Suppose \(\alpha\in j_E''\kappa^+.\) We claim that \(k_U(\alpha)\in \text{ran}(i)\). 
    Let \({\preceq}\) be a wellorder of \(\kappa\) of ordertype
    \(j_E^{-1}(\alpha)\).
    Then \(j_E({\preceq})\) has ordertype \(\alpha\). Note that \[k_U(j_E({\preceq})) = k_U(j_U({\preceq})\cap \rho) = j({\preceq})\cap i(\kappa) = i({\preceq})\]
    Thus \(k_U(\alpha)\), which is the ordertype of \(k_U(j_E({\preceq}))\) is equal to the ordertype of \(i({\preceq})\), which is in the range of \(i\).
    It follows that \((k_U\circ j_E)''\kappa^+\subseteq i''\kappa^+\subseteq A\).
    Since $k_U(\bar{A}) = A$, we conclude that \(j_E''\kappa^+\subseteq k_U^{-1}[A]\subseteq\bar A\).
\end{proof}
The proof of the following lemma, which was cited in the previous proposition, appears in \cite[Lemma 2.9]{HamkinsTall}:
\begin{lemma}\label{lma:extender_closure}
    Suppose \(E\) is an extender of length \(\rho\) with \(\crit(j_E) = \kappa\).
    If \({}^\kappa\rho\subseteq M_E\), then \({}^\kappa M_E\subseteq M_E\).
   In particular, if \(E\) is the extender of length \(\rho\) derived from an elementary embedding \(j : V\to M\) where \({}^\kappa M\subseteq M\), \(\cf(\rho) > \kappa\), and \(M\vDash \rho^\kappa = \rho\), then \({}^\kappa M_E\subseteq M_E\).\qed
\end{lemma}

Let us turn to the proof of Main Theorem \ref{mthorem: non-Galvin implies existence of ultrafilter}:
\begin{theorem}\label{Theorem: Main 1}
    Suppose that $\kappa$ is a non-Galvin cardinal. Then there exists a $\kappa$-complete ultrafilter $U$ over $\kappa$ such that $\neg \Gal{U,\kappa,\kappa^+}$. In particular, if $2^\kappa=\kappa^+$ then $U$ is non-Galvin.
\end{theorem}
\begin{proof}
    We use the notation of 
    \ref{Definition: non-Galvin cardinal}. As before, we can fix
    an ordinal \(\nu<j(\kappa)\)
    such that for some
    sequence \(\vec{A} = \langle A_\alpha\rangle_{\alpha < \kappa}\)
    such that \(A = j(\vec{A})_\nu\)
    and for some sequence 
    \(\vec{\kappa} = \langle \kappa_\alpha\rangle_{\alpha < \kappa}\),
    \(i(\kappa) = j(\vec{\kappa})_\nu.\)
    Let \(U=D(j,\nu)\) be the ultrafilter
    on \(\kappa\)
    derived from \(j\) using \(\nu\). Since $\crit(j)=\kappa$, $U$ is a $\kappa$-complete ultrafilter over $\kappa$.
    We will show
    $\neg \Gal{U,\kappa,\kappa^+}$.

    Let \(\langle f_\xi\rangle_{\xi < \kappa^+}\) denote the sequence
    of canonical functions on
    \(\kappa\) (see definition \ref{Definition: Canonical functions}).
    For \(\xi < \kappa^+\),
    define \[B_\xi = 
    \{\alpha < \kappa : f_\xi(\kappa_\alpha)\in A_\alpha\}\]
    Note that
    \(B_\xi\in U\)
    since \[j(B_\xi) = 
    \{\alpha < j(\kappa) : j(f_\xi)(j(\vec \kappa)_\alpha) \in j(\vec A)_\alpha\}\] 
    and 
    \[j(f_\xi)(j(\vec \kappa)_\nu) = j(f_\xi)(i(\kappa)) = 
    k(i(f_\xi))(i(\kappa))
    = i(\xi)\in A=j(\vec{A})_\nu\]
    The point here is that in \(N\),
    \(\vec g = i(\vec f)\)
    is the sequence of canonical functions
    on \(i(\kappa)\), and since
    \(\crit(k) = i(\kappa)\), by proposition \ref{Prop: Canonical functions},
    for any \(\eta < i(\kappa^+)\),
    \(k(g_\eta)(i(\kappa)) = \eta\).
    The fact that
    \(k(i(f_\xi))(i(\kappa))
    = i(\xi)\) follows from this observation
    when \(\eta = i(\xi)\)
    (and thus \(i(f_\xi) = g_{i(\xi)}\)).

    Suppose \(\sigma\subseteq \kappa^+\)
    and \(\bigcap_{\xi \in \sigma} B_\xi\in U\).
    We must show that \(|\sigma| < \kappa\). Since \(|A|^M < i(\kappa)\), it suffices to show that \(i(\sigma)\subseteq A\):
    then \(\text{ot}(i(\sigma)) <
    \text{ot}(A) <
    i(\kappa)\), and hence
    \(N\vDash \text{ot}(i(\sigma)) <
    i(\kappa)\), which by elementarity implies \(\text{ot}(\sigma) < \kappa\).
    
    The proof that
    \(i(\sigma)\subseteq A\) 
    is similar to the calculation 
    in the previous paragraph:
    Since \(\bigcap_{\xi \in \sigma} B_\xi\in U\),
    for all \(\eta \in j(\sigma)\),
    \(j(\vec f)_\eta(i(\kappa)) \in A\).
    Fix
    \(\xi \in i(\sigma)\), and we
    will prove that \(\xi\in A\).
    We have \(k(\xi)\in j(\sigma)\),
    so \(j(\vec f)_{k(\xi)}(i(\kappa))\in A\). But \(j(\vec f)_{k(\xi)} =
    k(g_\xi)\), hence
    \(k(g_\xi)(i(\kappa)) = \xi\). 
    It follows that \(\xi\in A\).
\end{proof}
\begin{remark}
    Note that in condition $(4)$ the definition \ref{Definition: non-Galvin cardinal} of non-Galvin cardinal it is important to work with $\kappa^+$ instead of $2^\kappa$ for there are no canonical functions in general up to $2^\kappa$. \end{remark}
\begin{remark}
    As proven in \cite{TomNatasha}, if $\kappa$ is $\kappa$-compact then there are $2^{2^\kappa}$-many \(\kappa\)-complete non-Galvin ultrafilters that extend the closed unbounded filter on \(\kappa\). On the other hand, assuming the Ultrapower Axiom and that every irreducible ultrafilter is Dodd sound, the least non-Galvin cardinal carries a unique non-Galvin ultrafilter that extends the closed unbounded filter on \(\kappa\).
    Under these assumptions, if \(\kappa\) carries distinct non-Galvin ultrafilters extending the closed unbounded filter, then the Ketonen least distinct such ultrafilters are precisely the least two extensions of the closed unbounded filter concentrating on singular cardinals
    (see the proof of Theorem \ref{Theorm: Main 4}). These ultrafilters are
    irreducible (and in fact are Mitchell points) by \cite[Corollary 8.2.13, Proposition 8.3.39]{GoldbergUA}. Therefore \(D_0 \lhd D_1\), so
    \(\kappa\) carries a non-Galvin ultrafilter in \(\text{Ult}(V,D_1)\), and so 
    \(\kappa\) is not the least non-Galvin cardinal.
    
\end{remark}

As a first upper bound for the non-Galvin cardinals we have the following:
\begin{theorem}\label{Theorem: Compact implies non-Galvin}
    If $\kappa$ is \(\kappa\)-compact, then
    \(\kappa\) is a non-Galvin cardinal.
\end{theorem}
\begin{proof}
    Let $U$ be a normal ultrafilter on $\kappa$.
     Since \(|P^{M_U}(P_\kappa(\kappa^+))| = 
    2^\kappa\), there is a transitive model $M$ with $$P^{M_U}(P_\kappa(\kappa^+))\subseteq M, \ |M|=2^\kappa$$ By Hayut's result \cite[Cor. 6]{YairSquare}, there is a transitive model $N$, an elementary embedding $j_0\colon M\rightarrow N$, with $\crit(j_0)=\kappa$ along with some $s\in N$, $s\subseteq j_0(\kappa)^+$ such that $j_0''\kappa^+\subseteq s$ with $|s|^N<j_0(\kappa)$. Define $\mathcal{W}$ the $\kappa$-complete ultrafilter on $P_\kappa(\kappa^+)$ derived from $j_0$ and $s$. Note that $\mathcal{W}$ is fine since $j_0''\kappa^+\subseteq s$ and it measures all the subsets of $P_\kappa(\kappa^+)$ in $M_U$. Let $j_{\mathcal{W}}\colon M_U\rightarrow M_{\mathcal{W}}$ be the ultrapower of $M_U$ by $\mathcal{W}$ defined in $V$, and $j\colon  V\to M_{\mathcal{W}}$ be the embedding $j = j_{\mathcal {W}}\circ j_U$. 
    Let $\lambda = j_\mathcal W(\kappa)<j(\kappa)$ and
    let $E$ be the extender of length $\lambda$ derived
    from $j$.
    \begin{claim}
        $E$ is also the extender of length $\lambda$ derived from $j_{\mathcal{W}}$
    \end{claim}
    \begin{proof}
        For any  $X\subseteq \kappa$ we have that $$ j(X)\cap \lambda=j_{\mathcal{W}}(j_U(X))\cap j_{\mathcal{W}}(\kappa)=j_{\mathcal{W}}(j_U(X)\cap\kappa)=j_{\mathcal{W}}(X).$$
        Thus for all $\alpha<\lambda$, $\alpha\in j(X)$ iff $\alpha \in j_{\mathcal{W}}(X)$.
    \end{proof}
    Finally let $i\colon V\to N_E$ be the ultrapower of $V$ by $E$ and $A = [\text{id}]_\mathcal W\in M_{\mathcal{W}}$. We claim that $i,j,A$ witness that $\kappa$ is a non-Galvin cardinal. Indeed, $i(\kappa)\geq\lambda$. To see that $i(\kappa)\leq \lambda$, we compute the ultrapower $i'$ of $M_U$ by $E$, and since $M_U$ is closed under $\kappa$-sequences, it follows that $i(\kappa)=i'(\kappa)$. By the previous claim, $j_{\mathcal{W}}$ also factors through $i'$ and thus $j_{\mathcal{W}}(\kappa)=k'(i'(\kappa))\geq i'(\kappa)=i(\kappa)$, as wanted. 

    By the usual argument about the derived extender, the factor map
    $k\colon N_E\rightarrow M_{\mathcal{W}}$ has critical point $i(\kappa)$ (see for example \cite[Lemma 20.29(ii)]{Jech2003}). Also, $M_{\mathcal{W}}\models |A|<j_{\mathcal{W}}(\kappa)=i(\kappa)$ and since $\mathcal{W}$ is fine, $j_{\mathcal{W}}''\kappa^+\subseteq A$. 
    \begin{claim}
        For every $\alpha<\kappa^+$,  $i(\alpha)=j_{\mathcal{W}}(\alpha)$.
    \end{claim}
    \begin{proof}
        
    Note that $i(U)\in N_E$ is a normal measure on $i(\kappa)$, let $X\in i(U)$ be any set, $k(X)\in j(U)=j_{\mathcal{W}}(j_U(U))$. Note that $j_{\mathcal{W}}(j_U(U))$ is generated by $j_{\mathcal{W}}'' j_U(U)$
    by Theorem 6 and Corollary 8 of \cite[Section 3]{Blass1970}.
    Therefore there is a set $Y\in j_U(U)$ such that $j_{\mathcal{W}}(Y)\subseteq k(X)$. Since $U$ is normal, there is a set $A\in U$ such that $j_U(A)\subseteq^* Y$ and $j(A)\subseteq^* k(X)$, which in turn implies that $i(A)\subseteq^* X$. Now we note that $i(A)\in R$, where $R$ is the $N_E$-ultrafilter (external) derived from $k$ and $j_{\mathcal{W}}(\kappa)$. We conclude that $i(U)\subseteq R$ and thus that $i(U)=R$ (as two $N_E$-ultrafilters). So $k$ factors through $j_{i(U)}$ and $k'\colon M_{i(U)}\rightarrow M_{\mathcal{W}}$ has critical point $>j_{\mathcal{W}}(\kappa)$ (since $k'(j_{\mathcal{W}}(\kappa))=k'([\text{id}]_{i(U)})=k(id)(j_{\mathcal{W}}(\kappa))=j_{\mathcal{W}}(\kappa)$).
    To conclude the claim, let $\alpha<\kappa^+$ and $f\colon \kappa\rightarrow\kappa$ be the canonical function such that $j_U(f)(\kappa)=\alpha$, then $$j_{\mathcal{W}}(\alpha)=j_{\mathcal{W}}(j_U(f)(\kappa))=j(f)(j_{\mathcal{W}}(\kappa))=k(i(f))(j_{\mathcal{W}}(\kappa))$$
    By elementarity, $i(f)\colon i(\kappa)\rightarrow i(\kappa)$ is the canonical function for $i(\alpha)$. Since $j_{i(U)}$ is the ultrapower by a normal ultrafilter over $j_{\mathcal{W}}(\kappa)$, we conclude that $$k(i(f))(j_{\mathcal{W}}(f))=k'(j_{i(U)}(i(f)))(j_{\mathcal{W}}(\kappa))=k'(i(\alpha))=i(\alpha)$$ as desired.
    \end{proof}
\end{proof}
\(\kappa\) is superstrong with an inaccessible target (which simply means that there is an elementary \(j : V\to M\) such that \(\crit(j) = \kappa\), \(V_{j(\kappa)}\subseteq M\), and
\(j(\kappa)\) is inaccessible in \(V\)),
then by the argument of \ref{Theorem: Partial sound}, \(\kappa\) is a non-Galvin cardinal.
Moreover, any subcompact cardinal is a limit of 
cardinals that are superstrong with an inaccessible target.

Hayut proved \cite{YairSquare} that $\kappa^+$-$\Pi^1_1$-subcompactness implies $\kappa$-compactness and he conjectures that these notions are equiconsistent\footnote{Since by the results of \cite{NeemanSteelSubcompact}, if there is a weakly iterable premouse with a $\kappa$-compact cardinal then in that inner model $\kappa$ is also $\kappa^+$-$\Pi^1_1$-subcompact cardinal.}. 
So morally speaking, $\kappa$-compact cardinals should be strictly greater than non-Galvin ultrafilters in the large cardinal hierarchy. In the next section, we will see that at least under UA this is the case. 
Finally, we establish the connection between Dodd soundness and non-Galvin cardinals: \begin{lemma}\label{Lemma: Soundness implies non-Galvin cardinal}
    Suppose that $U$ is a $\kappa$-complete non $p$-point $\lambda$-sound ultrafilter, and let $E$ be the $(\kappa,\lambda)$-extender derived from $j_U$ and $\lambda=\sup\{j_U(f)(\kappa)\mid f\colon\kappa\rightarrow\kappa\}$. Then $j_E''2^\kappa\in M_U$ and moreover $j_U,j_E,k_E$ and $j_E''2^\kappa$ witness that $\kappa$ is a non-Galvin cardinal.
\end{lemma}
\begin{proof}
    Derive the extender $E$ from $\lambda$ i.e $E=\l E_a\mid a\in[\lambda]^{<\omega}\r$ where $E_a$ is an ultrafilter over $[\kappa]^{|a|}$ defined by
    $$E_a=\{X\subseteq[\kappa]^{|a|}\mid a\in j(X)\}$$
    By $\lambda$-soundness of $U$, $E\in M_U$ and we let $i=j_E\colon M\rightarrow M_E$. Note that $j_E''P(\kappa)$ can be calculated in $M_U$ and therefore $j_E''P(\kappa)\in M_U$. Also, note that $j_E(\kappa)\geq\lambda$ and since $E\in M_U$, we must have that for every $a\in[\lambda]^{<\omega}$, $j_{E_a}(\kappa)<\lambda$ hence $j_E(\kappa)\leq\lambda$. We conclude that the critical point of the factor map $k_E\colon M_E\rightarrow M_U$ is $\lambda= j_E(\kappa)$. Finally, observe that $j_E''2^\kappa\in M_U$. To see this, simply note that $j_E\restriction On=(j_E)^{M_U}\restriction On$\footnote{This is since $M_U$ is closed under $\kappa$-sequences and thus the class of functions from $[\kappa]^{<\omega}$ to the ordinals is the same from the point of view of $V$ and $M_U$. Now both $j_E\restriction On$ and $(j_E)^{M_U}\restriction On$ are completely determined by those functions.} and therefore $j_E''2^\kappa=(j_E)^{M_U}{}''2^\kappa\in M_U$.
\end{proof}

\section{In the canonical inner models}
In this section, we work within the framework of $UA$ and ``every irreducible is Dodd sound." By results of Goldberg \cite{GoldbergUA} and Schlutzenberg \cite{Schlutz}, these assumptions hold in the  extender models $L[E]$.
Our first goal of this section is to prove Main Theorem \ref{mtheorem: characterization of sigma complete} regarding the characterization of $\sigma$-complete non-Galvin ultrafilters. To do that, we will need some preparatory results.   
\begin{theorem}[UA]\label{Theorem: irreducible non complete satisfy diamond}
    Suppose $\kappa$ is either successor or strongly inaccessible and
    \(U\) is a \(\kappa\)-irreducible non-$\kappa$-complete
    ultrafilter on \(\kappa\). Then \(\diamondsuit^{-}_{\text{thin}}(U)\).
\end{theorem}
    \begin{proof}
        By \cite[Theorems 8.2.22 and 8.2.23]{GoldbergUA},
        \(M_U\) is closed under \({<}\kappa\)-sequences and every \(A\in [M_U]^\kappa\)
        is covered by some \(B\in M_U\)
        such that \(|B|^{M_U} = \kappa\).
        By the assumptions of the theorem, \(U\) is not \(\kappa\)-complete and therefore $\crit(j_U)<\kappa$. Let $E^\kappa_\omega=\{\nu<\kappa\mid \cf(\nu)=\omega\}$,  define  the function $g:E^\kappa_\omega\rightarrow\kappa$ by
        $g(\nu)=\rho$ for the minimal measurable cardinal $\rho$ such that $j_U(\rho)>\nu$.  By \cite[Lemma 4.2.36]{GoldbergUA}, $g(\nu)$ is well defined and $g(\nu)\leq\nu$. Since $\cf(\nu)=\omega$, $g(\nu)<\nu$. By F\"{o}dor, there is an unbounded $S\subseteq E^\kappa_{\omega}$ and $\kappa^*<\kappa$ such that for every $\nu\in S$, $g(\nu)=\kappa^*$. In particular, $j_U(\kappa^*)\geq\kappa$. If $j_U(\kappa^*)>\kappa$, let $\gamma=\kappa^*$, otherwise $\kappa$ is a limit of $M_U$-strongly inaccessible cardinals. Let $\kappa^*<\gamma<\kappa$ be the least strongly inaccessible cardinal. In any case, $j_U(\gamma)>\kappa$ and since $M_U$ is closed under $<\kappa$-sequences, $\gamma$ is a strongly inaccessible cardinal in $V$.
        Therefore \(j_U''P_\kappa(\kappa)\)
        is covered by a set \(B\in M_U\)
        of cardinality less than \(j_U(\gamma)\).
        Let \(A= \{\bigcup S: S\in [B]^\kappa\cap M_U\}\).
        Then \(|A|^{M_U} < j_U(\gamma)\),
        and for any \(S\subseteq \kappa\),
        \(j_U(S)\cap \kappa_*\in A\)
        where \(\kappa_* =\sup j_U''\kappa\geq j_U(\gamma)\).
        Note that \(\kappa_* > j_U(f)(\alpha)\)
        for any \(f : \kappa\to \kappa\) and \(\alpha < \kappa_*\) and in particular there is no function $f:\kappa\rightarrow\kappa$ such that $j_U(f)(|A|^{M_U})\geq \kappa_*$.
        We conclude that \(A\) witnesses 
        \(\diamondsuit^-_{\textit{thin}}(U)\).
    \end{proof}
    
\begin{corollary}[UA]\label{Cor: ultrafilter on successor is Galvin}
If $U$ is a $\sigma$-complete ultrafilter over $\kappa^+$ then $\diamondsuit^-_{\text{thin}}(U)$ and in particular $U$ is non-Galvin.
 \end{corollary}
    \begin{proof}
        By \cite[Lemma 8.2.24]{GoldbergUA}, $U=\sum_{D}\l W_\xi\r_{\xi<\lambda_D}$ where $D$ is an ultrafilter over $\lambda_D<\kappa^+$, $\l W_\xi\r_{\xi<\lambda_D}$ is discrete and $M_D\models W=[\xi\mapsto W_\xi]_D$ is $j_D(\kappa^+)$-irreducible which cannot be $j_D(\kappa^+)$-complete. By the previous theorem, $M_D\models\diamondsuit^-_{\text{thin}}(W)$. Therefore, for $D$-almost all $\xi$, $\diamondsuit^-_{\text{thin}}(W_\xi)$ which by Lemma  \ref{Lemma: discrete limit of diamond minus}, implies that $\diamondsuit^-_{\text{thin}}(\sum_D\l W_\xi\r_{\xi<\lambda_D})$ holds.  
    \end{proof}    
    \begin{theorem}[UA]
        Assume that every irreducible is Dodd sound. If \(W\) is a $\kappa$-complete ultrafilter over $\kappa$, then the following are equivalent:
        \begin{enumerate}
        \item $W$ has the Galvin property.
        \item $\neg\diamondsuit^-_{\text{thin}}(W)$.
        \item $W$ is an $n$-fold sum of $\kappa$-complete $p$-points over $\kappa$
        \end{enumerate}    
    \end{theorem}
    \begin{proof}
        Let $W$ be $\kappa$-complete ultrafilter. If $W$ is an $n$-fold sum of $\kappa$-complete $p$-points then by Theorem \ref{Theorem: D-limit of p-points} $W$ has the Galvin property which by Theorem \ref{Theorem: diamond minus implies non-Galvin} implies $\neg\diamondsuit^-_{\text{thin}}(W)$.
    Let $W$ be a $\kappa$-complete ultrafilter over $\kappa$ which is not an $n$-fold sum of $\kappa$-complete $p$-points. Let $U\leq_{RF}W$ be irreducible, which exists since $W$ is nontrivial. If $U$ is not a $p$-point then by the assumptions of the theorem, $U$ is a non $p$-point ultrafilter Dodd sound over $\kappa$ and therefore by Lemma \ref{EquivDoddSound}, $\diamondsuit^*_{\text{thin}}(U)$ holds and thus also $\diamondsuit^-_{\text{thin}}(U)$. Since $U\leq_{RK} W$,  Lemma \ref{Lemma: Rudin-Keisler upword closure of diamond minus} applies, so we can conclude that $\diamondsuit^-_{\text{thin}}(W)$.   Hence we may restrict ourselves to the case where there is a $p$-point RF-below $W$ (and this $p$-point must be $\kappa$-complete). By \cite[Thm. 5.3.14]{GoldbergUA}, there is a $\leq_{RF}$-maximal 
        $U \leq_\text{RF} W$ that is an  $n$-fold sum of $\kappa$-complete  $p$-points over $\kappa$.
        Let $\langle W_\xi\r_{\xi <
        \kappa}$ be a discrete sequence
        with $W = U\text{-}\lim\l W_\xi\r_{\xi < \kappa}$. By the choice of $U$, the embedding $j_U:V\rightarrow M_U$ can be factored as a finite iterated ultrapower $$V = M_0\overset{j_{0,1}}{\longrightarrow}M_1\overset{j_{1,2}}{\longrightarrow}\cdots\overset{j_{n-1,n}}{\longrightarrow}M_n=M_U$$
        where in \(M_k\), $j_{k,k+1}$ is the ultrapower embedding associated to
        a $\kappa_{k}$-complete $p$-point \(U_k\) over $\kappa_{k}$ and $\kappa_k=j_{0,k}(\kappa)$. Also, denote by $Z_k$ the ultrafilter associated with $j_{0,k}$; i.e.,
        $$Z_k=U_0^\frown U_1^\frown\cdots^\frown U_{k-2}^\frown U_{k-1}$$
        For this notation, see Definition \ref{Definition:Frown}.
        Since $W_\xi$ is nonprincipal,
        there is an irreducible
        ultrafilter $D_\xi \leq_\text{RF} W_\xi$. Suppose that $D_\xi$ is $\rho_\xi$-complete uniform ultrafilter over $\delta_\xi$ for some $\rho_\xi\leq\delta_\xi\leq\kappa$. Note that $\sum_U D_\xi\leq_{RF}W$.
        Let $m$ be the least such that $\kappa_{m-1}<\delta^*:=[\xi\mapsto \delta_\xi]_U\leq\kappa_{m}$ where $\kappa_{-1}$ is defined to be $0$. Let $D^*=[\xi\mapsto D_\xi]_{U}$ is an $M_U$-ultrafilter over $\delta^*$. Note that $D^*\in M_{m}$ since $M_n\subseteq M_m$ and since $\crit(j_{m,n})=\kappa_{m}$ it is an $M_m$-ultrafilter. Moreover, $M_n^{\kappa_{m}}\cap M_{m}=M_n^{\kappa_{m}}\cap M_n$ and therefore $(j_{D^*})^{M_{m}}\restriction M_n=(j_{D^*})^{M_n}$. By elementarity of $j_{D^*}^{M_m}$, $j_{D^*}^{M_n}\circ j_{m,n}=j_{D^*}^{M_{m}}(j_{m,n})\circ j_{D^*}^{M_{m}}$ and we have that
                 \begin{equation}\label{representationOfPPoints}
            j_{D^*}^{M_n}\circ j_U=j_{D^*}^{M_{m}}(j_{m,n})\circ j_{D^*}^{M_{m}}\circ j_{0,m}.
                 \end{equation}
          \begin{claim}
            If $M_{m}\models D^*$ is not $\kappa_m$-complete, then $\diamondsuit^-_{\text{thin}}(W)$ holds.
        \end{claim} 
        \begin{proof}[Proof of claim.]
            Since $D_\xi$ is irreducible, by our assumption, it is a non $\kappa$-complete  Dodd sound ultrafilter. Note that in this case $m>0$, since if $m=0$, the $D^*$ must be $\kappa$-complete. Let us split unto cases:\begin{enumerate}
                \item[\underline{Case 1:}] If $\delta^*=\kappa_m$, then $D^*$ is a uniform ultrafilter on $\kappa_m$ and it must be $\kappa_m$-irreducible. By Theorem \ref{Theorem: irreducible non complete satisfy diamond} $M_m\models\diamondsuit^-_{\text{thin}}(D^*)$ holds. By Lemma \ref{Lemma: SumofDiamond}, we conclude that $\diamondsuit^-_{\text{thin}}(Z_m^\frown D^*)$ holds in $V$ 
                (see Definition \ref{Definition:Frown} for this notation),
                and hence by Lemma \ref{Lemma: Rudin-Keisler upword closure of diamond minus} $\diamondsuit^-_{\text{thin}}(W)$ follows as well.
                \item[\underline{Case 2:}]   Assume that $\delta^*<\kappa_m$. 
                  
                 \begin{enumerate}
                     \item[\underline{Case 2(b):}]
                 Assume $\crit(j_{D^*}^{M_m})>\kappa_{m-1}$. Note that the two step iteration ultrapower $j_{D^*}^{M_{m}}\circ j_{U_{m-1}}$ is given by
                a $\kappa_{m-1}$-complete $p$-point on $\kappa_{m-1}$ in $M_m$ (see \cite[Lemma 1.11]{SatInCan}), which contradicts the maximality of $U$. 
                
                \item [\underline{Case 2(c):}] Assume  $\crit(j_{D^*}^{M_{m}})\leq\kappa_{m-1}<\delta^*<\kappa_m$. Since $D^*$ is an irreducible uniform ultrafilter over $\lambda_{D^*}\geq\kappa_{m-1}^+$, $D^*$ is $\kappa_{m-1}^+$-irreducible and therefore by \cite[Theorem 8.2.22]{GoldbergUA}, $M_{D^*}$ is closed under $\kappa_{m-1}$-sequences which in turn implies that $P(\kappa_{m-1})\subseteq M_{D^*}$. By Lemma \cite[Lemma 4.2.36]{GoldbergUA}, $j^{M_m}_{D^*}(\kappa_{m-1})>\kappa_{m-1}$. Let $\lambda=j^{M_m}_{D^*}(\kappa_{m-1})$. We claim that ${U_{m-1}}^\frown D^*$ is $\lambda$-sound and that for every function $f:\kappa_{m-1}\rightarrow \kappa_{m-1}$, $j_{{U_{m-1}}^\frown D^*}(f)(\kappa_{m-1})<\lambda$ which by Corollary \ref{corollary: lambda sound and diamond minus} implies   that $\diamondsuit^-_{\text{thin}}({U_{m-1}}^\frown D^*)$. Indeed for any function $f:\kappa_{m-1}\rightarrow \kappa_{m-1}$, since $j^{M_m}_{D^*}(\kappa_{m-1})>\kappa_{m-1}$, \\$j^{M_m}_{D^*}(j_{U_{m-1}}(f))(\kappa_{m-1})=j_{D^*}^{M_m}(j_{U_{m-1}}(f)\restriction\kappa_{m-1})(\kappa_{m-1})=j_{D^*}^{M_m}(f)(\kappa_{m-1}),$ and $j_{D^*}^{M_m}(f):j_{D^*}^{M_m}(\kappa_{m-1})\rightarrow j_{D^*}^{M_m}(\kappa_{m-1})$. Hence $j_{D^*}^{M_m}(f)(\kappa_{m-1})<j_{D^*}^{M_m}(\kappa_{m-1})$.
                
                To see that ${U_{m-1}}^\frown D^*$ is $\lambda$-sound, derive the $(\kappa_{m-1},\lambda)$-extender $E$ from $j_{D^*}^{M_m}$ inside $M_m$. Note that $E$ is also the $(\kappa_{m-1},\lambda)$-extender derived from $j_{D^*}\circ j_{m-1,m}$ since for $\alpha<j_{D^*}^{M_m}(\kappa_{m-1})$ we have that:\\ $\alpha\in j_{D^*}^{M_m}(j_{m-1,m}(X))\cap j_{D^*}^{M_m}(\kappa_{m-1})$ iff  $\alpha\in j_{D^*}^{M_m}(j_{m-1,m}(X)\cap \kappa_{m-1})$ iff $ \alpha\in j^{M_m}_{D^*}(X)$. 
                
                Now $D^*$ is a uniform ultrafilter over $\delta^*>\kappa_{m-1}$, hence we have that $j^{M_m}_{D^*}(\kappa)<[id]_{D^*}$ and since $D^*$ is Dodd sound we have that $E\in (M_{D^*})^{M_m}$. In particular, $\{j_E(X)\mid X\subseteq\kappa_{m-1}\}\in (M_{D^*})^{M_{m}}$ where $j_E:M_{m-1}\rightarrow M_E$. Let $k_E:M_E\rightarrow (M_{D^*})^{M_{m}}$ be the factor map. It follows that $\crit(k_E)= j_{D^*}^{M_m}(\kappa_{m-1})$. Finally, note that $j_{{U_{m-1}}^\frown D^*}(X)\cap j^{M_m}_{D^*}(\kappa_{m-1})=j_E(X)$, hence $$\{j_{{U_{m-1}}^\frown D^*}(X)\cap j^{M_m}_{D^*}(\kappa_{m-1})\mid X\subseteq \kappa_{m-1}\}\in (M_{D^*})^{M_m}$$ as desired. We conclude that $M_{m-1}\models\diamondsuit^-_{\text{thin}}({U_{m-1}}^\frown D^*)$. By Lemma \ref{Lemma: non discrete} $\diamondsuit^-_{\text{thin}}({Z_{m-1}}^\frown {U_{m-1}}^\frown D^*$, and this ultrafilter is Rudin-Keisler below $W$.\qedhere\end{enumerate} 
                
            \end{enumerate}  \end{proof}
            By the claim, we may assume that for $M_m\models D^*$ is $\kappa_m$-complete over $\kappa_m$. It follows again that in $M_m$, $D^*$ cannot be a $p$-point, as this would contradict the maximality of \(U\), recalling that $\sum_U D_\xi\leq_{RF}W$ and that this ultrafilter $\sum_UD_\xi$ can be represented as an $n+1$-fold sum of $\kappa$-complete $p$-points by (\ref{representationOfPPoints}). Since $D^*$ is irreducible in $M_m$, $M_m\models D^*$ is Dodd-sound and non $p$-point. By Lemma \ref{EquivDoddSound}  $M_m\models\diamondsuit^*_{\text{thin}}(D^*)$ holds. In particular, $\diamondsuit^-_{\text{thin}}(D^*)$ holds. In any case, Lemma \ref{Lemma: SumofDiamond} applies to conclude that $\diamondsuit^-_{\text{thin}}({Z_m}^\frown D^*)$ holds, and since this ultrailter is $RK$-below $W$, lemma \ref{Lemma: Rudin-Keisler upword closure of diamond minus} ensures that $\diamondsuit^-_{\text{thin}}(W)$ holds.
\end{proof}\begin{theorem}[UA]\label{theorem:Gal-char} Assume that every irreducible ultrafilter is
    Dodd sound. For every $\sigma$-complete ultrafilter $W$ over $\kappa$ the following are equivalent:
    \begin{enumerate}
        \item $W$ has the Galvin property.
        \item $\neg\diamondsuit^-_{\text{thin}}(W)$.
        \item $W$ is the $D$-sum of $n$-fold sums of $\kappa$-complete $p$-points over $\kappa$ and $D$ is a $\sigma$-complete ultrafilter on $\lambda<\kappa$.
    \end{enumerate} 
    \end{theorem}
    \begin{proof}
        The proof that $(3)\Rightarrow(1)\Rightarrow (2)$ is in the previous theorem. It remains to prove that $\neg\diamondsuit^-_{\text{thin}}(W)$ implies that $W$ is a $D$-sum of $n$-fold sums of $\kappa$-complete $p$-points over $\kappa$. Equivalently, let us prove the contrapositive, suppose that \(W\) is a $\sigma$-complete  ultrafilter over $\kappa$ which is not an $n$-fold sum of $p$-points. 
         Now let us move to the general case, suppose that $W$ is just $\sigma$-complete. By \cite[Lemma 8.2.24]{GoldbergUA}, there is a countably complete ultrafilter $D\leq_{RF} W$ on $\lambda<\kappa$ such that if $W=D\text{-}\lim\l W_\xi\r_{\xi<\lambda}$ then $M_D\models Z=[\xi\mapsto W_\xi]_D$ is $j_D(\kappa)$-irreducible. If $Z$ is not $j_D(\kappa)$-complete then by Theorem \ref{Theorem: irreducible non complete satisfy diamond}. $M_D\models \diamondsuit^-_{\text{thin}}(Z)$ and therefore $W=D\text{-}\lim \l W_\xi\r_{\xi<\lambda}$ will also satisfy $\diamondsuit^-_{\text{thin}}$ by Lemma \ref{Lemma: discrete limit of diamond minus}. If $Z$ is $j_D(\kappa)$-complete, then $Z$ is a $j_D(\kappa)$-complete ultrafilter which is not a $D'$-sum of $n$-fold sums of $p$-points and we fall into the first case where we assumed that $W$ was $\kappa$-complete (inside $M_D$ and replacing $\kappa$ by $j_D(\kappa)$). We conclude that $\diamondsuit^-_{\text{thin}}(Z)$ holds and again, it follows from that $\diamondsuit^-_{\text{thin}}(W)$ holds.
    \end{proof}

    Next, we turn to the proof of Main Theorem \ref{mtheorem: extending the club filter under UA}.
\begin{theorem}[UA]\label{Theorm: Main 4}
    Assume that every irreducible ultrafilter
    is Dodd sound.
    Suppose $\kappa$ is 
    an uncountable cardinal
    that carries a $\kappa$-complete non-Galvin ultrafilter.
    Then the Ketonen least non-Galvin
    $\kappa$-complete
    ultrafilter on $\kappa$
    extends the closed unbounded filter.
\end{theorem}
\begin{proof}
    We claim that in this context, 
    the Ketonen least non-Galvin ultrafilter \(U\)
    is equal to the Ketonen least
    ultrafilter \(W\)
    on a regular cardinal \(\delta\) extending the closed
    unbounded filter and concentrating on singular cardinals.
    First, note that
    \(W\) is irreducible by 
    \cite[Corollary 8.2.12]{GoldbergUA}. 
    \begin{claim}\label{Claim: W is complete}
        $W$ is $\delta$-complete
    \end{claim}
    \begin{proof}[Proof of Claim \ref{Claim: W is complete}.] Suppose towards a contradiction that $W$ is not $\delta$-complete and let $\mu=crit(j_W)<\delta$. Since $W$ is Dodd sound, $j_W$ is a $2^{<\delta}$-supercompact embedding (see \cite[Lemma 4.3.4]{GoldbergUA}),
    and so \(j_W\) witnesses that $\mu$ is $2^{<\delta}$-supercompact. In particular,
    \(\mu\) is \(2^\mu\)-supercompact,
    and therefore every \(\mu\)-complete filter on \(\mu\) extends to a \(\mu\)-complete ultrafilter. This yields a $\mu$-complete ultrafilter $W'$ on \(\mu\) extending the closed unbounded filter
    on \(\mu\) adjoined with the set of singular cardinals less than \(\mu\).
    Since $\mu<\delta$, it follows that $W'<_{\Bbbk}W$ (see \cite[Lemma 3.3.15]{GoldbergUA}) contradicting the minimality of $W$.
    \end{proof}
    \noindent\textit{End of proof of Theorem \ref{Theorm: Main 4}.} Note that \(W\) is not a $p$-point since \(W\) extends the closed unbounded filter but is not normal; therefore by 
    Corollary \ref{Corollary: nonppoint+DoddSound implies non-Galvin}, \(W\) is non-Galvin, and hence \(U\) is below \(W\)
    in the Ketonen order.

    Conversely, since \(U\)
    is the Ketonen least non-Galvin ultrafilter, by Theorem \ref{Theorem: Galvin Improvment},
    \(U\) is 
    irreducible and not a $p$-point.
    Without loss of generality, we can assume
    that \(U\) is Dodd sound.
    Moreover, \(U\) is a 
    \(\gamma\)-complete ultrafilter
    on \(\gamma\) for some 
    measurable cardinal \(\gamma\).
    
    Let \(\lambda = \sup \{j_U(f)(\gamma) \mid
    f\colon \gamma\rightarrow\gamma\}\).
    Since \(U\) is not a \(p\)-point,
    \(\lambda \leq [\text{id}]_U\).
    Since \(U\)
    is Dodd sound, 
    \(\{j_U(A)\cap \lambda : A\subseteq \gamma\}\in M_U\), which implies
    \[\{j_U(f)\cap (\lambda\times \lambda) \mid f : \gamma\to \gamma\}\in M_U\] and hence 
    \(\{j_U(f)(\gamma) \mid
    f\colon \gamma\rightarrow\gamma\}\in M_U\),
    which implies that \(M_U\) satisfies \(\cf(\lambda) \leq 2^\gamma\).
    
    Let \(D\) be the 
    ultrafilter
    on \(\gamma\)
    derived from \(j_U\) using
    \(\lambda\). 
    Then \(D\) is below \(U\)
    in the Ketonen order.
    Since
    \(\cf^{M_U}(\lambda) \leq 2^\gamma\),
    \(D\) concentrates on singular cardinals. Moreover, for any
    \(f\in \gamma^\gamma\),
    \(\lambda\) is closed under 
    \(j_U(f)\) --- that is, \(j_U(f)[\lambda]\subseteq \lambda\)
    --- so \(D\) concentrates on the set
    of closure points of \(f\). It follows that \(D\) extends the
    closed unbounded filter.
    Therefore \(W\) is below 
    \(D\) in the Ketonen order,
    so by the transitivity of the Ketonen order, \(W\)
    is below \(U\) in the Ketonen order. It follows that
    \(U = W\) as claimed.
    This implies that \(U\) extends the club filter, which proves the theorem.
\end{proof}
Let us turn our attention to the non-Galvin cardinals. Main Theorem \ref{mtheorem: non-Galvin ulder UA is optimal}, which we now prove, shows that the existence of a non-Galvin cardinal is exactly the large cardinal assumption needed to conclude the existence of non-Galvin ultrafilters in an inner model.
\begin{theorem}[UA]\label{Theorem: Main 2}
    Assume that every irreducible ultrafilter
    is Dodd sound. If there is a 
    $\kappa$-complete non-Galvin ultrafilter
    on an uncountable cardinal $\kappa$,
    then there is a non-Galvin cardinal.
\end{theorem}
\begin{proof}
    Let $W$ be a non-Galvin ultrafilter on $\kappa$. By Theorem  \ref{Theorem: Factorization into irreducibles}, $W$ is Rudin-Keisler equivalent to an $n$-fold sum of irreducible ultrafilters. By Theorem \ref{Theorem: Galvin Improvment}, it is impossible that all these ultrafilters are $p$-points (even on measure one sets) so $\kappa$ must carry an irreducible ultrafilter $U$ which is not a $p$-point. By our assumption, every irreducible is Dodd sound. Since $U$ is a $\kappa$-complete, non $p$-point, Dodd sound ultrafilter, Lemma \ref{Lemma: Soundness implies non-Galvin cardinal} applies, and we conclude that $\kappa$ is a non-Galvin cardinal.  
\end{proof}
\begin{proposition}[UA]\label{Proposition: UA implies compact above non-Galvin}
    If \(\kappa\)
    is \(\kappa\)-compact
    and no cardinal \(\nu < \kappa\)
    is \(\kappa\)-supercompact,
    then \(\kappa\)
    a limit of non-Galvin cardinals.
    \begin{proof}
    Since \(\kappa\)
    is \(\kappa\)-compact, a theorem
    of Kunen \cite[Lemma 3]{KUNEN1971} implies that
    for every \(\xi < (2^{\kappa})^+\),
    there is a countably complete
    ultrafilter
    \(U\) on \(\kappa\)
    such that \(j_U(\xi) > \xi\).
    Let \(U_\xi\) denote the Ketonen least
    such ultrafilter.
    By \cite[Lemma 7.4.34]{GoldbergUA} and
    \cite[Proposition 8.3.39]{GoldbergUA}, \(U_\xi\)
    is a \textit{Mitchell point}:
    for any ultrafilter \(W<_{\Bbbk} U\),
    \(W\) lies below \(U\)
    in the Mitchell order.

     Since \(\kappa\) is strongly inaccessible, there is an \(\omega\)-club \(C\subseteq(2^\kappa)^+\)
    such that for all \(\xi\in C\), for all
    countably complete
    ultrafilters \(D\)
    of rank less than \(\xi\) in the Ketonen order,
    \(j_D(\xi) = \xi\). 
    For \(\xi \in C\),
    \(U_\xi\) is a uniform irreducible ultrafilter on \(\kappa\), and so it follows from \cite[Theorem 8.2.23]{GoldbergUA} that
    \(U_\xi\) witnesses 
    \(\crit(j_{U_\xi})\) is 
    \({<}\kappa\)-supercompact.
    Since \(\kappa\) is measurable,
    it follows that
    \(\crit(j_{U_\xi})\) is \(\kappa\)-supercompact,
    and so by the assumptions of the proposition, 
    \(\crit(j_{U_\xi}) = \kappa\).
    In other words \(U_\xi\) is
    \(\kappa\)-complete.

    Now let \(W\)
    witness that \(\kappa\)
    is a non-Galvin cardinal. 
    Fix \(\xi \in C\)
    larger than the Ketonen rank of \(W\).
    Then \(W\) is below 
    \(U_\xi\) in the Mitchell order, and so \(\kappa\) is non-Galvin
    in \(M_{U_\xi}\).
    It follows that \(\kappa\) is a limit of non-Galvin cardinals.
    \end{proof}
\end{proposition}

In particular,
the least cardinal \(\kappa\)
that is \(\kappa\)-compact
is larger than the least non-Galvin cardinal assuming UA.\footnote{It should be provable from UA that any cardinal \(\kappa\) that is \(\kappa\)-compact is a limit of non-Galvin cardinals. Here there are two cases. If \(\kappa\)
is a limit of cardinals \(\gamma\) that are \(\kappa\)-compact, then 
each of these cardinals \(\gamma\) is \(\gamma\)-compact, so \(\kappa\)
is a limit of non-Galvin cardinals. If \(\kappa\)
is not a limit of \(\kappa\)-compact cardinals, one would like to show, as above, that there is a non-Galvin ultrafilter \(W\) on \(\kappa\) that is below some \(\kappa\)-complete ultrafilter on \(\kappa\) in the Mitchell order. The issue is that it is unclear how to show that the Mitchell order on \(\kappa\)-complete ultrafilters has rank \((2^\kappa)^+\) if some \(\nu < \kappa\) is \(\kappa\)-supercompact.}
\section{Open problems}
\begin{question}
    Is it consistent that there is a $\kappa$-complete uniform ultrafilter $U$  over $\kappa$ satisfying the Galvin property that is not an $n$-fold sum of  $\kappa$-complete $p$-points over $\kappa$?
\end{question}
Recently, Gitik gave a positive answer to this question, thus our characterization of ultrafilters with the Galvin property cannot be proved in ZFC. The following question seems more plausible for a positive answer in ZFC:
\begin{question}
    Is every uniform $\kappa$-complete ultrafilter $U$ over $\kappa^+$ non-Galvin, i.e., $\neg \text{Gal}(U,\kappa^+,\kappa^{++})$ holds? 
\end{question}
Under UA, the answer is positive by Corollary \ref{Cor: ultrafilter on successor is Galvin}.  

\begin{question}
    Does a non-Galvin cardinal entail the existence of a non-Galvin ultrafilter which extends the club filter?
\end{question}
By Main Theorem \ref{mthorem: non-Galvin implies existence of ultrafilter}, a non-Galvin cardinal entails the existence of a non-Galvin ultrafilter. Assuming UA and that every irreducible is Dodd sound, Main Theorem \ref{mtheorem: extending the club filter under UA} shows that a non-Galvin cardinal also entails the existence of $\kappa$-complete non-Galvin ultrafilter which extends the club filter. 

\begin{question}
    Does every fine normal ultrafilter over $P_\kappa(\kappa^+)$ satisfy $\text{Gal}(U,\kappa,2^{\kappa^{+}})$?
\end{question}
The answer would be interesting even under $UA$. This is the first step toward answering the more general problem:
\begin{question}
    Characterize the Tukey-top ultrafilters on $\kappa$ with respect to $\lambda<\kappa$ assuming UA plus every irreducible is Dodd sound.
\end{question}
 \begin{question}
    Is there a similar characterization under $UA$ for $\sigma$-complete ultrafilters with the Galvin property over singular cardinals?
\end{question}
We believe that such a characterization exists and that similar methods to those appearing in this paper should be useful. 

In the absence of GCH we have the following questions which are open:
\begin{question}
    If we replace $i''\kappa^+$ by $i''2^\kappa$ in the definition of non-Galvin cardinal, do we get a $\kappa$-complete ultrafilter such that $\neg\text{Gal}(U,\kappa,2^\kappa)$?
\end{question}
More generally:
\begin{question}
    Is it consistent that there is a $\kappa$-complete ultrafilter $U$ such that $\neg \text{Gal}(U,\kappa,\kappa^+)$ but $\text{Gal}(U,\kappa,2^\kappa)$?
\end{question}
The result of this paper resolves these two questions under UA plus every irreducible is Dodd sound.

 The following two questions address the assumptions in the main theorems of this paper.

\begin{question}
    Is it consistent that there is a cardinal $\kappa$ which is $\kappa^+$-supercompact and that every irreducible ultrafilter is Dodd sound?
\end{question}
\begin{question}
    Does UA imply that every irreducible ultrafilter is Rudin-Keisler equivalent to a Dodd sound ultrafilter?
\end{question}

 Let us conclude this paper with a diamond-like principle which is a reasonable candidate to be equivalent to non-Galvin ultrafilters.
Such a principle would be valuable as there is no known formulation of the Galvin property in terms of the ultrapower. This would be also interesting from the point of view of the Tukey order since this order involves functions which typically have  domains of size $2^\kappa$, and thus not available in the ultrapower.  
\begin{definition}
    We say that $\diamondsuit^-_{\text{par}}(U)$ holds if and only if there is $A\in M_U$, $\lambda$ and $\l X_i\r_{i<2^\kappa}\subseteq P(\kappa)$ such that:
    \begin{enumerate}
        \item $\{j_U(X_i)\cap\lambda\mid i<2^\kappa\}\subseteq A$.
        \item there is no function $f:\kappa\rightarrow\kappa$ such that $j_U(f)(|A|^{M_U})\geq\lambda$.
    \end{enumerate}
\end{definition}
The argument of Theorem \ref{Theorem: diamond minus implies non-Galvin} can be adjusted to conclude that $\diamondsuit^-_{\text{par}}(U)$ implies that $U$ is non-Galvin. 
\begin{question}
Is $\diamondsuit^-_{\text{par}}(U)$ equivalent to $U$ being non-Galvin?    
\end{question}
The next question seeks an analogous result on $\omega$ to the one of this paper: 
\begin{question}
    Is it consistent that every ultrafilter on $\omega$ which is not Tukey-top is an $n$-fold sum of $p$-points?
\end{question}
\subsection*{Acknowledgment}
The authors would like to thank the referee of this paper for their clever remarks and contribution to the presentation of the current version of this paper. They would also like to thank Natasha Dobrinen for providing valuable corrections regarding the theory of the Tukey order on ultrafilters over a countable set. Finally, we would like to thank Moti Gitik for insightful discussions and comments. 

\bibliographystyle{amsplain}
\bibliography{ref}
\end{document}